\documentclass[11pt,a4paper]{amsart}
\usepackage[a4paper,twoside,centering]{geometry}
\usepackage[arrow,matrix]{xy}
\usepackage{amssymb}

\title{Refined Coxeter polynomials}
\author{Sefi Ladkani}

\address{Department of Mathematics, University of Haifa, Mount Carmel,
Haifa 3498838, Israel}
\email{ladkani.math@gmail.com}

\DeclareMathOperator{\End}{End}
\DeclareMathOperator{\Ext}{Ext}
\DeclareMathOperator{\GL}{GL}
\DeclareMathOperator{\HH}{HH}
\DeclareMathOperator{\Hom}{Hom}
\DeclareMathOperator{\modf}{mod}
\DeclareMathOperator{\per}{per}
\DeclareMathOperator{\Roots}{Roots}
\DeclareMathOperator{\sgn}{sgn}

\newcommand{\cC}{\mathcal{C}}
\newcommand{\cD}{\mathcal{D}}
\newcommand{\eps}{\varepsilon}
\newcommand{\gL}{\Lambda}
\newcommand{\wh}{\widehat}
\newcommand{\wt}{\widetilde}
\newcommand{\bQ}{\mathbb{Q}}
\newcommand{\bR}{\mathbb{R}}
\newcommand{\bS}{\mathbb{S}}
\newcommand{\bZ}{\mathbb{Z}}

\newtheorem*{theorem*}{Theorem}
\newtheorem{theorem}{Theorem}[section]
\newtheorem{prop}[theorem]{Proposition}
\newtheorem{lemma}[theorem]{Lemma}
\newtheorem{cor}[theorem]{Corollary}

\theoremstyle{definition}
\newtheorem{defn}[theorem]{Definition}
\newtheorem{rem}[theorem]{Remark}
\newtheorem{example}[theorem]{Example}
\newtheorem{notat}[theorem]{Notation}

\numberwithin{equation}{section}

\begin{document}

\begin{abstract}
Coxeter polynomials are important homological invariants
that are defined for a large class of finite-dimensional
algebras.
It is of particular interest to develop methods to compute
these polynomials.
We define the notion of insertion of a poset into a triangular algebra
at a vertex of its quiver and show that its Coxeter polynomial is
controlled in a uniform way by two polynomials attached to the
poset that we call refined Coxeter polynomials.
Several properties of these polynomials are discussed.
Applications include
new symmetry properties for Coxeter polynomials of ordinal sums of posets,
constructions of new algebras of cyclotomic type
and interlaced towers of algebras.
\end{abstract}

\maketitle

\section{Introduction}

Coxeter polynomials play prominent role in the representation theory
of finite dimensional algebras, as demonstrated by the
survey articles~\cite{delaPena94,LenzingdelaPena08};
they are discrete invariants of derived equivalence, some of their
coefficients encode various homological information~\cite{Happel97}
and they provide links between the representation theory of quivers
and spectral graph theory~\cite{ACampo76,CvRoSi10}
as well as Lie theory~\cite{Bourbaki}.

It is an interesting and challenging problem to compute the 
Coxeter polynomials for various families of finite-dimensional algebras
given in combinatorial terms. This has been done for the families
of canonical, supercanonical and extended canonical algebras, see
the survey~\cite{LenzingdelaPena08}. A method to compute the
Coxeter polynomials of algebras obtained by gluing two quivers
with relations at a vertex~\cite{Boldt95} allows to recursively
compute them for path algebras of trees. Coxeter polynomials
of one-point extensions have been discussed in~\cite{Happel09}.

In this article we focus on triangular algebras.
A \emph{triangular algebra} is a quotient of a path algebra of a quiver
without oriented cycles by an admissible ideal. To such algebra~$\gL$
one can associate its \emph{Coxeter polynomial}
$\phi_\gL(x) = \det(xC_{\gL} + C_{\gL}^T)$, where $C_{\gL}$ is the
Cartan matrix of $\gL$. It is a monic, self-reciprocal polynomial.
Given a vertex $v$ in the quiver of $\gL$
and a finite partially ordered set (poset) $S$, we introduce
the notion of \emph{insertion} of $S$ to $\gL$ at $v$.
Roughly speaking, one inserts $S$ at $v$ and modifies appropriately the
ideal of relations to obtain a new triangular algebra
$\gL \xleftarrow{v} S$, we refer to Definition~\ref{def:insert} for
a precise description. This construction generalizes previous
constructions in the literature, such as that of supercanonical
algebras~\cite{LenzingdelaPena04} or lexicographic sums of
posets~\cite{Ladkani08a}.

Simple examples show that the Coxeter polynomials of the algebras
$\gL$, $\gL^-=(1-e)\gL(1-e)$ (where $e$ is the primitive idempotent
corresponding to $v$) and (the incidence algebra of) $S$ are not
sufficient to determine that of $\gL \xleftarrow{v} S$, see
Example~\ref{ex:pXvS}. However, we can prove the following result,
see Theorem~\ref{t:refined}.
\begin{theorem*}
Let $S$ be a poset. There exist polynomials
$\phi^0_S(x), \phi^1_S(x) \in \bZ[x]$ such that
for any triangular algebra $\gL$ and any vertex $v$ in its quiver,
\[
\phi_{\gL \xleftarrow{v} S}(x) =
\phi_{\gL}(x) \cdot \phi^0_S(x) + \phi_{\gL^-}(x) \cdot \phi^1_S(x) .
\]
\end{theorem*}

The polynomials $\phi^0_S$ and $\phi^1_S$
are called the \emph{refined Coxeter polynomials} of $S$.
They can be expressed in terms
of the Coxeter polynomials of $S$ and the poset $\wh{S}$ obtained from
$S$ by adding a unique maximal element (Proposition~\ref{p:refbyCox}),
and thus can be seen as refinement of the Coxeter polynomial of $S$.
We discuss various properties of these polynomials.
Among these, we obtain a generalization of a formula for the
Coxeter polynomial for coalescence of trees
(Corollary~\ref{cor:insertS1}).

We present three applications of refined Coxeter polynomials.
The first concerns ordinal sums of posets.
It is known~\cite{Ladkani08a}
that if $X_1$ and $X_2$ are two posets,
then the incidence algebras of the ordinal sums $X_1 \oplus X_2$
and $X_2 \oplus X_1$ are derived equivalent, and hence
their Coxeter polynomials coincide.
However, there
are examples~[loc.\ cit.]
showing that for three posets,
the incidence algebras of the ordinal sums
$X_1 \oplus X_2 \oplus X_3$ and $X_2 \oplus X_1 \oplus X_3$ are
in general not derived equivalent. Nevertheless, it is interesting
to ask whether their Coxeter polynomials still coincide.
We derive a formula for the Coxeter polynomial of an ordinal sum of
posets in terms of their refined Coxeter polynomials
(Proposition~\ref{p:ordsum}) and
deduce that the Coxeter polynomial of any ordinal sum of posets
is independent on the order of summands (Corollary~\ref{cor:ordsum}).

The second application concerns algebras of cyclotomic type.
These are triangular algebras whose Coxeter polynomials are
products of cyclotomic polynomials.
A systematic study of these algebras has been initiated
in the paper~\cite{delaPena14}, which surveys several families of such algebras, including the path algebras of Dynkin and Euclidean quivers,
the canonical algebras~\cite{Lenzing99,Ringel84},
some supercanonical~\cite{LenzingdelaPena04} and
extended canonical~\cite{LenzingdelaPena11} algebras,
algebras that are fractionally Calabi-Yau, and hence 
by~\cite{HerschendIyama11} $n$-representation-finite algebras.

We show that the polynomial $\phi^0_S$ vanishes
for any poset $S$ whose incidence algebra
is a path algebra of an orientation of an Euclidean diagram of type
$\wt{A}$ (Theorem~\ref{t:Atilde}). As a consequence we produce new
families of algebras of cyclotomic type; consider the smallest
class $\cC$ of posets containing the empty set and satisfying
\begin{itemize}
\item[]
If $X$ is a poset with an element $v \in X$ such that
$X \setminus \{v\}$ belongs to $\cC$,
\\
then $X \xleftarrow{v} S$ belongs to $\cC$
for any poset $S$ of type $\wt{A}$.
\end{itemize}
Then the incidence algebra of any poset in $\cC$ is of cyclotomic type
and more generally, if $\gL$ is a triangular algebra and $v$ is a vertex
in its quiver such that $\gL^-$ is of cyclotomic type, then so is
$\gL \xleftarrow{v} S$ for any $S$ in $\cC$ (Corollary~\ref{cor:cyclo}).

The last application deals with interlaced towers of algebras 
in the sense of~\cite{delaPena15}. After observing that any Coxeter
polynomial $\phi(x)$ can be represented as
$\phi(x^2) = x^{\deg q} q(x+x^{-1})$ for some polynomial $q \in \bZ[x]$
(Corollary~\ref{cor:represent}, see~\cite{LenzingdelaPena09} for the
notion of representability), we can rephrase the original definition;
a sequence $\gL_1, \gL_2, \dots, \gL_n$ of triangular algebras
of length $n \geq 3$
is an \emph{interlaced tower of algebras} if the
roots of the polynomial representing $\phi_{\gL_1}$ are
real and simple,
$\phi_{\gL_{i+1}}(x) =
(x+1) \cdot \phi_{\gL_i}(x) - x \cdot \phi_{\gL_{i-1}}(x)$
for all $1 < i < n$ and
$\deg \phi_{\gL_{i+1}} = \deg \phi_{\gL_i} + 1$
for all $1 \leq i < n$.
One seeks interlacing properties of the roots of the polynomials
representing $\phi_{\gL_i}$ similar to 
the interlacing properties of eigenvalues of symmetric
matrices used in spectral graph theory~\cite{CvRoSi10}.

We observe that if $\gL$ is a triangular algebra such that all the roots
of the polynomial representing $\phi_\gL$ are real and simple,
then for any vertex $v$
the sequence of insertions $\gL \xleftarrow{v} \overrightarrow{A_i}$
forms an interlaced tower of algebras, where $\overrightarrow{A_i}$
denotes a chain with $i$ elements (Proposition~\ref{p:tower}).
This allows to construct several examples refuting claims made
in Theorem~2 and Theorem~5.5 of~\cite{delaPena15} 
on the Mahler measure
of Coxeter polynomials for interlaced towers of algebras.

The paper is organized as follows. 
In Section~\ref{sec:Coxeter} we provide some general background
on Coxeter polynomials for finite-dimensional algebras and
survey their basic properties.
Section~\ref{sec:Refined} is devoted to the theory of refined
Coxeter polynomials. The notion of insertion of a poset into a triangular
algebra at a vertex is introduced and a formula for the Coxeter polynomial
of such insertion in terms of the refined Coxeter polynomials
of the poset is presented. Further properties of these polynomials
are discussed.
Applications to ordinal sums of posets, algebras of cyclotomic type
and interlaced towers of algebras are described in Section~\ref{sec:Apps}.

\subsection*{Acknowledgement}

This is an extended version of my talk given at the ICRA 2020 conference
that took place virtually in November 2020. I thank the organizers for
the opportunity to give a talk at the conference and for the
invitation to contribute to this volume.
I thank the anonymous referee for carefully reading the manuscript
and providing many valuable comments and suggestions which helped improving the
presentation.

\section{Coxeter polynomials}
\label{sec:Coxeter}

\subsection{The Coxeter polynomial as a numerical derived invariant}

Let $K$ be a field and let $\cD$ be a triangulated $K$-linear category with
suspension functor $\Sigma$. Assume that for any $M, N \in \cD$,
\begin{itemize}
\item
$\dim_K \Hom_{\cD}(M,N) < \infty$ (i.e.\ $\cD$ is \emph{$\Hom$-finite}),

\item
$\Hom_{\cD}(M, \Sigma^i N)$ vanishes for all but finitely many 
elements $i \in \bZ$.
\end{itemize}
Then one can consider the expression
\[
\langle M, N \rangle =
\sum_{i \in \bZ} (-1)^i \dim_K \Hom_{\cD}(M, \Sigma^i N)
\]
which induces a well-defined bilinear form with integer values
on the \emph{Grothendieck group} $K_0(\cD)$, known as the
\emph{Euler form}.
Recall that
$K_0(\cD)$ is the quotient of the free abelian group
generated by the set of isomorphism classes
$[M]$ of objects $M \in \cD$ by its
subgroup generated by the expressions
$[M] - [L] - [N]$ for all the distinguished triangles
$L \to M \to N \to \Sigma L$ in $\cD$.

Any $K$-linear triangulated functor $F \colon \cD \to \cD'$ induces a map
$[F] \colon K_0(\cD) \to K_0(\cD')$.
If $F$ is an equivalence, the map $[F]$ is an isomorphism
which is moreover an isometry of the corresponding Euler forms, i.e.\
$\langle [F]u, [F]v \rangle = \langle u, v \rangle$ for
any $u,v \in K_0(\cD)$.

Assume that $K_0(\cD)$ is a \emph{lattice}, i.e.\ a free abelian group of
finite rank. By choosing a basis, we may identify $K_0(\cD)$ with
$\bZ^n$ for some integer $n \geq 1$ and represent the Euler form by a
matrix $C \in M_n(\bZ)$. If $\cD'$ is an equivalent
triangulated category and $C'$ is a matrix representing the Euler form
on $K_0(\cD')$ with respect to some basis, then the matrices
$C$ and $C'$ are \emph{congruent}, that is, there exists a matrix
$P \in \GL_n(\bZ)$ such that $C' = P C P^T$.

In order to extract a numerical invariant for $\cD$,
we form the matrix $xC + C^T$ whose entries lie in the polynomial ring $\bZ[x]$,
and consider its determinant, which is a polynomial with integer coefficients.

\begin{lemma}
Let $p(x) = \det(xC+C^T)$. Then 
$p(x)$ is invariant under triangulated equivalence
and moreover $p(x) = x^n p(x^{-1})$.
\end{lemma}
\begin{proof}
The first claim follows since $p(x)$ is invariant under
congruence of matrices, as
$x(PCP^T) + (PCP^T)^T = P(xC+C^T)P^T$ and $(\det P)^2=1$
for any $P \in \GL_n(\bZ)$. For the second claim,
\begin{align*}
p(x) &= \det(xC + C^T) = \det(xC + C^T)^T =
\det(xC^T+C) = x^n \det(C^T + x^{-1}C) 
\\
&= x^n p(x^{-1}) .
\qedhere
\end{align*}
\end{proof}

Assume now that $C$ is \emph{invertible} in $M_n(\bQ)$ (equivalentely,
$\det C \neq 0$). In this case one can extract a factor of $\det C$
and write
\[
\det(xC+C^T) = \det C \cdot \det(xI_n + C^{-1} C^T)
\]
so the remaining factor becomes a characteristic polynomial, namely
that of the matrix $-C^{-1} C^T \in M_n(\bQ)$. This motivates the
next definition.

\begin{defn} \label{def:CoxC}
The \emph{Coxeter polynomial} $\phi_{\cD}$ of $\cD$ is the
characteristic polynomial of $-C^{-1}C^T$, where $C$ is any matrix
representing the Euler form on $K_0(\cD)$.
\end{defn}

The above considerations show that the Coxeter polynomial is well-defined
when $K_0(\cD)$ is a lattice and the Euler form is non-degenerate
as a bilinear form on the vector space $K_0(\cD) \otimes_{\bZ} \bQ$.
We list some of its basic properties:
\begin{itemize}
\item
$\phi_{\cD}$ is a monic polynomial with rational coefficients which is
\emph{self-reciprocal}, i.e.\
$\phi_{\cD}(x) = x^n \phi_{\cD}(x^{-1})$.
In other words, $\phi_{\cD}(x) = x^n + a_{n-1}x^{n-1} + \dots + a_1 x + a_0$
where $a_0=1$ and $a_i=a_{n-i}$ for all $1 \leq i < n$.

\item
$\phi_{\cD}$ in an invariant of triangulated equivalence, i.e.\
if $\cD \simeq \cD'$ then $\phi_{\cD} = \phi_{\cD'}$.
\end{itemize}

It is possible to define a Coxeter polynomial even if the matrix $C$ is not
invertible, provided that some homological assumptions on the category $\cD$
are made.
Recall that a \emph{Serre functor} $\nu \colon \cD \to \cD$ is an additive
auto-equivalence with bi-functorial isomorphisms
$\Hom_{\cD}(M,\nu N) \simeq \Hom_{\cD}(N,M)^*$
for all $M, N \in \cD$, where $-^*$ denotes the dual vector space,
see~\cite[\S3]{BondalKapranov89} or \cite[\S1]{BondalOrlov01}.
A Serre functor $\nu$ is triangulated~\cite[Prop.~1.4]{BondalOrlov01},
hence it induces a linear map $[\nu] \colon K_0(\cD) \to K_0(\cD)$.
We may call the map $-[\nu]$ the \emph{Coxeter transformation} on $K_0(\cD)$,
for the following reason:
if $\cD$ is Krull-Schmidt, the existence of a Serre functor $\nu$ is
equivalent to the existence of Auslander-Reiten
triangles~\cite[\S I]{ReitenVdB02}, and the Auslander-Reiten translation
on $\cD$ is then given by $\Sigma^{-1} \nu$ whose corresponding map
on $K_0(\cD)$ equals $-[\nu]$.
Relations between the Auslander-Reiten translation (in categories of
modules) and the Coxeter transformation can be traced back
to~\cite[\S5]{Gabriel80} and \cite[\S2.4]{Ringel84}.

\begin{defn} \label{def:CoxSerre}
The \emph{Coxeter polynomial} $\phi_{\cD}$ of $\cD$ with a Serre
functor $\nu$ is the characteristic polynomial of any matrix
representing the Coxeter transformation $-[\nu]$.
\end{defn}

According to this definition, the coefficients of $\phi_{\cD}$ are
integers.
Now observe that $\langle M, \nu N \rangle = \langle N, M \rangle$.
Thus if $C$ is the matrix 
representing the Euler form and $S$ is the matrix representing the
Coxeter transformation with
respect to the same basis of $K_0(\cD)$, then $C \cdot S = -C^T$.
This shows that Definition~\ref{def:CoxSerre} agrees with
Definition~\ref{def:CoxC} when 
$\cD$ has a Serre functor and the Euler form is non-degenerate.

The category $\cD$ is \emph{fractionally Calabi-Yau}~\cite{Kontsevich98}
if it has a Serre
functor $\nu$ and there are integers $p$ and $q \geq 1$ such that
$\nu^q \simeq \Sigma^p$. Being fractionally Calabi-Yau implies that
the Coxeter transformation has finite order and hence the
Coxeter polynomial $\phi_{\cD}$ is a product of cyclotomic
polynomials~\cite[Theorem~2.10]{delaPena14}.
Note that these implications are strict.

\subsection{The Coxeter polynomial of a finite-dimensional algebra}

The discussion in the preceding section applies in particular to
triangulated categories associated with finite-dimensional algebras.
We recall the basic facts and refer to the book~\cite{Happel88}
for more details.

Let $\gL$ be a finite-dimensional algebra over a field $K$. Denote
by $\modf \gL$ the category of finitely generated right $\gL$-modules
and by $\cD^b(\modf \gL)$ its bounded derived category, whose objects are
bounded complexes of finitely generated $\gL$-modules. It is 
a triangulated category, containing as a full triangulated
subcategory the category $\per \gL$ of perfect complexes whose objects are
bounded complexes of finitely generated projective $\gL$-modules.

The triangulated category $\per \gL$ satisfies the assumptions made in the
beginning of the preceding section. Moreover, if
$P_1, P_2, \dots, P_n$ form a complete collection of
pairwise non-isomorphic indecomposable projective modules in 
$\modf \gL$, then the Grothendieck group $K_0(\per \gL)$ is isomorphic
to $\bZ^n$, with a basis given by the classes of
$P_1, P_2, \dots, P_n$. The matrix of the Euler form with respect
to this basis is given by the \emph{Cartan matrix} $C_\gL$ of $\gL$, 
whose entries are defined by $(C_\gL)_{ij} = \dim_K \Hom_{\gL}(P_i, P_j)$.

\begin{defn}
The \emph{Coxeter polynomial} $\phi_\gL$ of $\gL$ is the
Coxeter polynomial of $\per \gL$, whenever the latter is defined.
\end{defn}

As discussed in the previous section, we distinguish two cases:
\begin{enumerate}
\renewcommand{\theenumi}{\alph{enumi}}
\item
If $\det C_\gL \neq 0$, then $\phi_\gL(x) = \det(xI_n + C_\gL^{-1} C_\gL^T)$
according to Definition~\ref{def:CoxC}.

\item
If $\gL$ is \emph{Iwanaga-Gorenstein}, that is,
its $K$-dual $\gL^*$ has a finite projective resolution as a right
$\gL$-module and $\gL_{\gL}$ has a finite injective resolution, 
then $\per \gL$ has a Serre functor given by the Nakayama functor
$- \stackrel{\mathbf{L}}{\otimes}_{\gL} \gL^*$ and
Definition~\ref{def:CoxSerre} applies.
As noted in the previous section, this definition agrees with the
former when $\det C_\gL \neq 0$.
\end{enumerate}

\begin{example}
If $\gL$ is a symmetric algebra, that is $\gL \simeq \gL^*$ as
$\gL$-bimodules, then the identity functor on $\per \gL$ is a Serre
functor and $\phi_\gL(x) = (x+1)^n$
regardless of $C_\gL$ being invertible or not.
\end{example}

If $S_1, S_2, \dots, S_n$ form a complete collection of the
pairwise non-isomorphic simple $\gL$-modules,
numbered such that $P_i$ is the projective cover of $S_i$,
then their classes in the Grothendieck group of $\cD^b(\modf \gL)$ form
a basis and $[P_j] = \sum_{i=1}^{n} [P_j : S_i] [S_i]$
for any $1 \leq j \leq n$, where $[P_j : S_i]$ denotes the multiplicity
of $S_i$ in a composition series of $P_j$. 
Hence $(C_\gL)_{ij} = [P_j : S_i] \cdot \dim_K \End_{\gL}(S_i)$
for any $1 \leq i,j \leq n$.

In the sequel we will focus on the case where the algebra $\gL$ has finite
global dimension, that is, any module in $\modf \gL$ has a finite projective
resolution.
The inclusion of $\per \gL$ into
$\cD^b(\modf \gL)$ is then an equivalence of triangulated categories,
hence the matrix of multiplicities $([P_j : S_i])_{i,j=1}^n$ is invertible
over $\bZ$ and thus the Cartan matrix $C_\gL$ is invertible over $\bQ$.
If all the endomorphism algebras $\End_{\gL}(S_i)$ are one-dimensional
over $K$, then $C_{\gL}$ is even invertible over $\bZ$. This happens when
$K$ is algebraically closed, or more generally, when $\gL$ can be written
as a quotient
of a path algebra of a quiver by an ideal generated by linear combinations
of parallel paths (see the next section for precise definitions).

Moreover, if $\gL$ has finite global dimension then $\cD^b(\modf \gL)$ has
a Serre functor, so both definitions of the Coxeter polynomial apply, hence
$\phi_{\gL}(x) = \det(xI_n + C_{\gL}^{-1} C_{\gL}^T)$ is well-defined
and equals the characteristic polynomial of the Coxeter transformation.
Note that the matrix $c_{\gL}=(c_{ij})$ of the Euler form with respect to
the basis consisting of the classes of simple modules is given by
\[
c_{ij} = \sum_{r \geq 0} (-1)^r \dim_K \Ext^r_{\gL}(S_i, S_j) ,
\]
so $\phi_{\gL}(x) = \det(xI_n + c^{-1}_{\gL}c^T_{\gL})$ as well.

Another homological interpretation of the Coxeter polynomial is provided
by the following result of Happel~\cite{Happel97} describing,
when the field $K$ is algebraically closed and the algebra $\gL$
has finite global dimension, the coefficient $a_1$ of $x$ in $\phi_{\gL}(x)$,
which is minus the trace of the Coxeter transformation,
in terms of the dimensions of the Hochschild cohomology groups of $\gL$ as
$a_1 = \sum_{i \geq 0} (-1)^i \dim_K \HH^i(\gL)$.

\section{Refined Coxeter polynomials}
\label{sec:Refined}

\subsection{Posets and their incidence algebras}

For background on posets we refer to the book~\cite{Stanley}.
The derived categories of their incidence algebras are studied
in~\cite{Ladkani08a}.

Throughout, by a \emph{poset} we will always mean a finite partially
ordered set.
Let $X$ be a poset and let $K$ be a field.
Consider the $K$-linear vector space spanned by all the symbols $e_{xy}$
for the pairs $(x,y) \in X \times X$ with $x \leq y$. Define the
product of two symbols $e_{xy} \cdot e_{x'y'}$ as $e_{xy'}$ if $x'=y$ and
zero otherwise,
and extend it by linearity. This equips the above vector space with
a structure of a $K$-algebra, known as the \emph{incidence algebra} of $X$
over $K$ and denoted by $KX$.

The incidence algebra is a finite-dimensional algebra of finite global
dimension. If $x \in X$, then the right $KX$-module
$P_x = e_{xx} KX$ is an indecomposable projective module, and this gives
a bijection between the isomorphism classes of such modules and the
elements of $X$. One has
\[
\dim_K \Hom_{KX}(P_x, P_y) =
\begin{cases}
1 & \text{if $y \leq x$,} \\
0 & \text{otherwise}
\end{cases}
\]
hence the Cartan matrix of $KX$ equals the transpose of the incidence
matrix of $X$. It follows that the Cartan matrix and the Coxeter polynomial
of the algebra $KX$ are independent on the ground field $K$,
and therefore we will denote them by $C_X$ and $\phi_X$ respectively,
see also~\cite[\S3.4]{Ladkani08b}.

Recall that the \emph{Hasse quiver} $Q_X$ of $X$ is
the quiver whose vertices are the elements of $X$ and whose arrows
are defined by the rule that for any $x, y \in X$ there is an arrow
$x \to y$ if and only if $y$ \emph{covers} $x$, that is, $x < y$ and
there is no $z \in X$ such that $x < z < y$.
The Hasse quiver does not have any oriented
cycles, and the incidence algebra $KX$ is the quotient
of the path algebra $KQ_X$ by its ideal generated by the
differences $p-q$ for all pairs of paths $p,q$ which start at the same
vertex and end at the same vertex.

\subsection{Insertion of posets}

\begin{defn}
Let $(X, \leq_X)$ be a poset and $v \in X$. The \emph{insertion} of a poset
$(S, \leq_S)$ to $X$ at $v$, denoted by $X \xleftarrow{v} S$, is the poset
whose underlying set is the disjoint union $(X \setminus \{v\}) \cup S$
with the partial order $\leq$ given by
\begin{align*}
x \leq x' \Longleftrightarrow x \leq_X x', &&
s \leq s' \Longleftrightarrow s \leq_S s', &&
x \leq s  \Longleftrightarrow x \leq_X v, &&
s \leq x  \Longleftrightarrow v \leq_X x
\end{align*}
for any $x, x' \in X \setminus \{v\}$ and $s, s' \in S$.
\end{defn}

The next lemma records some basic properties of the insertion operation.

\begin{lemma} \label{l:insert} {\ }
\begin{enumerate}
\renewcommand{\theenumi}{\alph{enumi}}
\item
Let $v \in X$. Then 
$X \xleftarrow{v} \varnothing = X \setminus \{v\}$ and
$X \xleftarrow{v} \{\bullet\} \simeq X$.

\item
If $v, v'$ are two distinct elements of a poset $X$ and
$S, S'$ are posets, then
\[
(X \xleftarrow{v} S) \xleftarrow{v'} S' \simeq
(X \xleftarrow{v'} S') \xleftarrow{v} S .
\]

\item
If $X, Y, S$ are posets and $v \in X$, $u \in Y$, then
\[
Y \xleftarrow{u} (X \xleftarrow{v} S) \simeq
(Y \xleftarrow{u} X) \xleftarrow{v} S .
\]
\end{enumerate}
\end{lemma}

Many classical operations on posets can be interpreted in terms of
the insertion construction. 
Let $n \geq 1$ and let $(X_i,\leq_i)$ be posets for $1 \leq i \leq n$.
Recall that there are two natural partial orders on the
disjoint union $\bigsqcup_{i=1}^n X_i$. In order to define them, 
let $x, y$ be two elements in $\bigsqcup_{i=1}^n X_i$ and write
$x \in X_i$ and $y \in X_j$ for some $1 \leq i, j \leq n$.

The first order, called the \emph{disjoint union} and denoted
by $X_1 + X_2 + \dots + X_n$, is defined by
$x \leq y$ if and only if $i=j$ and $x \leq_i y$. The second, called
the \emph{ordinal sum} and denoted by
$X_1 \oplus X_2 \oplus \dots \oplus X_n$, is defined by
$x \leq y$ if and only if $i<j$ or $i=j$ and $x \leq_i y$.

\begin{lemma} \label{l:ordsum}
Let $X_1, X_2, \dots X_n$ be posets. Then
\begin{align*}
X_1 + X_2 + \dots + X_n
\simeq
(( \dots (((
\xymatrix@1@=1pc{
{\bullet_1} & {\bullet_2} & {\dots} & {\bullet_n}
}) \xleftarrow{1} X_1) \xleftarrow{2} X_2)
\dots ) \xleftarrow{n} X_n) ,
\\
X_1 \oplus X_2 \oplus \dots \oplus X_n
\simeq
(( \dots (((
\xymatrix@1@=1pc{{\bullet_1} \ar[r] & {\bullet_2} \ar[r] & {\dots} \ar[r] & 
{\bullet_n}}
) \xleftarrow{1} X_1) \xleftarrow{2} X_2)
\dots ) \xleftarrow{n} X_n) ,
\end{align*}
where we represented the posets we insert at by their Hasse quivers.
\end{lemma}

One would like to compute the Coxeter polynomial of the insertion
$X \xleftarrow{v} S$ in terms of quantities related to $X$, $v$ and $S$.
The next example, involving path algebras of Dynkin quivers,
shows that the knowledge of $X$, $v$ and the
Coxeter polynomial of $S$ is not sufficient to determine that
of $X \xleftarrow{v} S$.

\begin{example} \label{ex:pXvS}
$\phi_S = \phi_{S'}$ but $\phi_{X \xleftarrow{v} S} \neq
\phi_{X \xleftarrow{v} S'}$.

Consider the posets below given by their Hasse quivers.
For clarity, the posets $S,S'$ and the element $v$ of $X$
we insert at are colored black while the other element of $X$
is colored white.
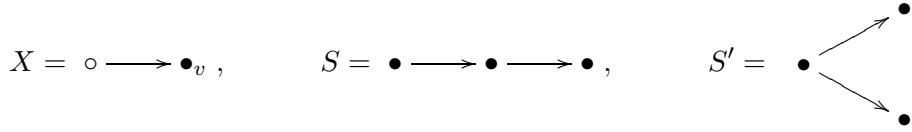
\begin{align*}
X = \xymatrix{{\circ} \ar[r] & {\bullet_v}} &,&
S = \xymatrix{
{\bullet} \ar[r] & {\bullet} \ar[r] & {\bullet}
} &,&
S' = 
\begin{array}{c}
\xymatrix@=0.8pc{
&& {\bullet} \\ {\bullet} \ar[urr] \ar[drr] \\ && {\bullet}
}
\end{array}
\end{align*}
Then the Hasse quivers of $S$ and $S'$ are orientations of the Dynkin
diagram $A_3$ and $\phi_{S}(x) = \phi_{S'}(x) = x^3+x^2+x+1$. However,
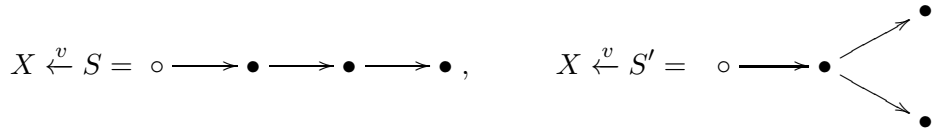
\begin{align*}
X \xleftarrow{v} S =
\xymatrix{
{\circ} \ar[r] & {\bullet} \ar[r] & {\bullet} \ar[r] & {\bullet}
}
&,&
X \xleftarrow{v} S' =
\begin{array}{c}
\xymatrix@=0.8pc{
&& && {\bullet} \\ {\circ} \ar[rr] && {\bullet} \ar[urr] \ar[drr] \\
&& && {\bullet}
}
\end{array}
\end{align*}
so $\phi_{X \xleftarrow{v} S}(x) = x^4+x^3+x^2+x+1$ (Dynkin type $A_4$)
but $\phi_{X \xleftarrow{v} S'}(x) = x^4+x^3+x+1$ (Dynkin type $D_4$).

\end{example}

\subsection{Insertion of a poset into a triangular algebra}

We generalize the operation of insertion of posets to the realm of
triangular algebras using the language of quivers with relations.
We start by introducing some notations.
For a quiver $Q$, denote by $Q_0$ its set of vertices and by $Q_1$ its
set of arrows. For any arrow $\alpha \in Q_1$, let
$h(\alpha), t(\alpha) \in Q_0$ denote the vertices that $\alpha$ starts
and ends at, and extend this definition to paths. Two paths $p,q$ in $Q$
are \emph{parallel} if $h(p)=h(q)$ and $t(p)=t(q)$.

\begin{defn} 
A \emph{triangular algebra} is a quotient $\gL=KQ/I$
where $Q$ is a quiver without oriented cycles
and $I$ is an ideal generated by linear combinations of parallel paths
of lengths at least $2$.
\end{defn}

The isomorphism classes of indecomposable projective $\gL$-modules
correspond bijectively to the vertices of $Q$. If $\gL$ is triangular,
these vertices can be labeled by the numbers $1,2,\dots,|Q_0|$ such that
if $i \geq j$, there are no arrows from vertex $j$ to vertex $i$ and
thus the simple module $S_i$ does not occur as a composition factor of the
kernel of the cover map $P_j \to S_j$.
It follows (by induction on $j$) that each module $S_j$ has a finite
projective resolution and hence the algebra $\gL$ has finite global dimension.
Moreover, with respect to such labeling, the Cartan matrix $C_\gL$ is upper
triangular with $1$ on the main diagonal. In particular, $\det C_\gL=1$ and
the Coxeter polynomial can be written as $\phi_\gL(x) = \det(xC_\gL+C^T_\gL)$.

Let $S$ be a nonempty poset. 
An element $s \in S$ is \emph{minimal} if there is no $s' \in S$ such that
$s' < s$. Denote by $S_{\min}$ the set of all minimal elements in $S$,
and define similarly the set $S_{\max}$ of maximal elements in $S$.
An element $s \in S$ is a source (sink) in the Hasse quiver of $S$ 
if and only if it belongs to $S_{\min}$ ($S_{\max}$, respectively).

We are now ready to define what an insertion is.
\begin{defn} \label{def:insert}
Let $\gL = KQ/I$ be a triangular algebra and let $v \in Q_0$.
Let $S$ be a nonempty poset with Hasse quiver $Q_S$.
The \emph{insertion} of $S$ into $\gL$ at $v$, denoted $\gL \xleftarrow{v} S$,
is the quotient $K\wt{Q}/\wt{I}$ where the quiver $\wt{Q}$ and the ideal
$\wt{I}$ are defined as follows.

\begin{itemize}
\item
To get $\wt{Q}_0$ from $Q_0$, 
remove the vertex $v$ and add all the elements of $S$, that is,
\[
\wt{Q}_0 = (Q_0 \setminus \{v\}) \cup S .
\]

\item
For each arrow $\alpha \in Q_1$ ending at $v$ and each $s \in S_{\min}$,
define a new arrow $(\alpha,s)$ which starts at $h(\alpha)$
and ends at $s$.
Similarly, for each arrow $\beta \in Q_1$ starting at $v$ and each
$s \in S_{\max}$, define a new arrow $(s,\beta)$ which starts at $s$
and ends at $t(\beta)$. Then
\begin{align*}
\wt{Q}_1  = & \left\{ \gamma \in Q_1 : h(\gamma), t(\gamma) \neq v \right\}
\cup (Q_S)_1 \\
 \cup & \left\{ (\alpha,s) : \text{$\alpha \in Q_1$ such that $t(\alpha)=v$
and $s \in S_{\min}$} \right\} \\
 \cup & \left\{ (s,\beta) : \text{$\beta \in Q_1$ such that $h(\beta)=v$
and $s \in S_{\max}$} \right\} .
\end{align*}

\noindent
If $p$ is a path in $Q$ ending at $v$ and $s \in S_{\min}$, let
$(p,s)$ be the path in $\wt{Q}$ obtained from $p$ by replacing its last arrow
$\alpha$ by $(\alpha,s)$. Define similarly $(s,p)$
for a path $p$ in $Q$ starting at $v$ and $s \in S_{\max}$.

\item
The ideal $\wt{I}$ is generated by the following elements:
\begin{enumerate}
\renewcommand{\labelenumi}{\theenumi.}
\item
Modification of relations starting or ending at $v$:
\begin{enumerate}
\item \label{it:ins:smin}
The elements $\sum c_p (p,s)$
where $s \in S_{\min}$ and
$\sum c_p p \in I$ is a linear combination of parallel paths ending at $v$.

\item \label{it:ins:smax}
The elements $\sum c_p (s,p)$
where $s \in S_{\max}$ and
$\sum c_p p \in I$ is a linear combination of parallel paths starting at $v$.
\end{enumerate}
(here, $c_p \in K$ are scalars).

\item
Commutativity relations arising from $Q_S$:
\begin{enumerate}
\item \label{it:ins:commS}
$q'-q''$ for any two parallel paths $q',q''$ in $Q_S$.

\item \label{it:ins:commvout}
For any arrow $\alpha \in Q_1$ ending at $v$ and any two paths $q',q''$ in
$Q_S$ such that $h(q'), h(q'') \in S_{\min}$ and $t(q')=t(q'')$,
\[
(\alpha, h(q'))q' - (\alpha, h(q''))q''
\]

\item \label{it:ins:commvin}
For any arrow $\beta \in Q_1$ starting at $v$ and any two paths $q',q''$ in
$Q_S$ such that $t(q'), t(q'') \in S_{\max}$ and $h(q')=h(q'')$,
\[
q' (t(q'), \beta) - q'' (t(q''), \beta)
\]

\item \label{it:ins:commvthru}
For any $\alpha, \beta \in Q_1$ such that $t(\alpha)=h(\beta)=v$
and any two paths $q',q''$ in $Q_S$
such that $h(q'), h(q'') \in S_{\min}$ and $t(q'), t(q'') \in S_{\max}$,
\[
(\alpha, h(q')) q' (t(q'), \beta) - (\alpha, h(q'')) q'' (t(q''), \beta)
\]
\end{enumerate}

If $p$ is a path in $Q$ passing through $v$ (but not starting or ending at
$v$), write $p=p'p''$ as a composition of non-trivial paths $p',p''$ with
$t(p')=h(p'')=v$. Choose a maximal path $q$ in $Q_S$, i.e.\
$h(q) \in S_{\min}$ and $t(q) \in S_{\max}$, define
$\wt{p} = (p',h(q)) q (t(q),p'')$ and note that
the value of $\wt{p}$ modulo $\wt{I}$ is independent on the choice of
$q$. If $p$ does not pass through $v$, set $\wt{p}=p$.

\item \label{it:ins:vthru}
The elements
\[
\sum c_p \wt{p}
\]
where $\sum c_p p \in I $ is a linear combination of parallel paths that
do not start nor end at $v$.
\end{enumerate}
\end{itemize}
\end{defn}

By construction, the quiver $\wt{Q}$ has no oriented cycles and hence
the insertion $\gL \xleftarrow{v} S$ is also a triangular algebra.

\begin{notat}
If $\gL = KQ/I$ and $v \in Q_0$, denote by $\gL^-$ the algebra
$(1-e_v) \gL (1-e_v)$ where $e_v \in \gL$ is the primitive idempotent
corresponding to $v$.
\end{notat}

\begin{rem} \label{rem:insert01}
If the poset $S$ has a unique minimal element $\wh{0}$ and a unique maximal
element $\wh{1}$ then to form the insertion of $S$ to $\gL$ at $v$ one simply
``inserts'' the quiver $Q_S$ with its full commutativity relations at the
place of $v$.
Namely, arrows and relations in $Q$ that ended in $v$ now end at $\wh{0}$,
arrows and relations that started in $v$ now start at $\wh{1}$; 
there are no commutativity relations in addition to those on $Q_S$;
and any occurrence of $\alpha \beta$ with $t(\alpha)=h(\beta)=v$
inside a path appearing in a relation is replaced by
$\alpha q \beta$ where $q$ is any path in $Q_S$ from $\wh{0}$ to $\wh{1}$.
\end{rem}

In the next examples, consider a triangular algebra $\gL$ and
a vertex $v$ in its quiver.
Denote by $\alpha_1,\dots,\alpha_k$ be the arrows ending at $v$ and by
$\beta_1,\dots,\beta_l$ the arrows starting at~$v$.

\begin{example}[Insertion of a chain] \label{ex:insertAn}
Let $S=\overrightarrow{A_n}$ be a chain with $n$ elements
$v_1 < v_2 < \dots < v_n$. Its Hasse quiver is a linear orientation of the
Dynkin diagram $A_n$.
Since each of the sets $S_{\min}, S_{\max}$
consists of just one element, Remark~\ref{rem:insert01} applies.
The quiver of the insertion $\gL \xleftarrow{v} \overrightarrow{A_n}$
is obtained by replacing $v$ with the oriented line
$v_1 \to v_2 \to \dots \to v_n$ as shown in Figure~\ref{fig:insert}.
Relations in $\gL$ that ended at $v$ now end at $v_1$,
and relations that started at $v$ now start at $v_n$. Finally,
for relations involving paths passing through $v$,
replace by $\alpha_i \rho_1 \dots \rho_{n-1} \beta_j$ any occurrence of
$\alpha_i \beta_j$ inside a path which is part of such a relation.
\end{example}

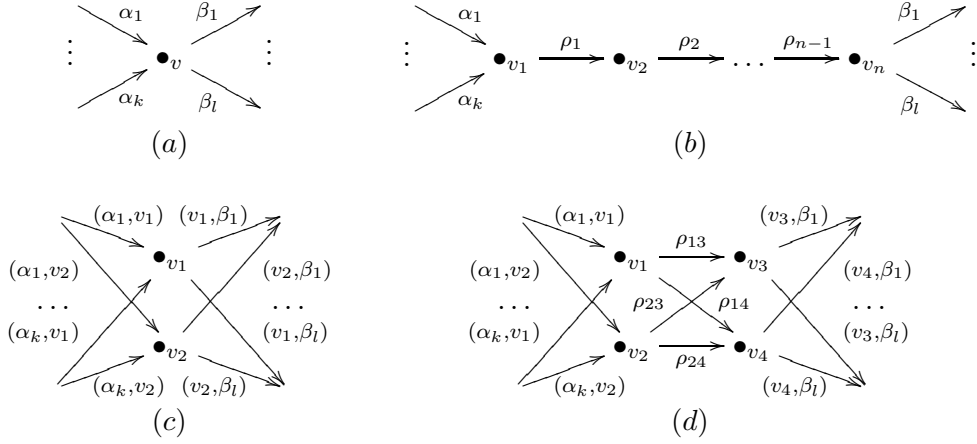
\begin{figure}
\[
\begin{array}{ccc}
\xymatrix@R=0.5pc{
{} \ar[dr]^{\alpha_1} & & {} \\
{\vdots} & {\bullet_v} \ar[ur]^{\beta_1} \ar[dr]_{\beta_l} & {\vdots} \\
{} \ar[ur]_{\alpha_k} & & {}
}
&&
\xymatrix@R=0.5pc{
{} \ar[dr]^{\alpha_1} & & & & & {} \\
{\vdots} & {\bullet_{v_1}} \ar[r]^{\rho_1} &
{\bullet_{v_2}} \ar[r]^{\rho_2} & {\dots} \ar[r]^{\rho_{n-1}} &
{\bullet_{v_n}} \ar[ur]^{\beta_1} \ar[dr]_{\beta_l} & {\vdots} \\
{} \ar[ur]_{\alpha_k} & & & & & {}
}
\\
(a) && (b) \\ \\
\xymatrix@R=0.5pc{
{} \ar[dr]^{(\alpha_1,v_1)} \ar[dddr]_(.3){(\alpha_1,v_2)} && {} \\
& {\bullet_{v_1}} \ar[ur]^{(v_1,\beta_1)} \ar[dddr]^(.7){(v_1,\beta_l)} \\
{\dots} && {\dots} \\
& {\bullet_{v_2}} \ar[dr]_{(v_2,\beta_l)} \ar[uuur]_(.7){(v_2,\beta_1)} \\
{} \ar[ur]_{(\alpha_k,v_2)} \ar[uuur]^(.3){(\alpha_k,v_1)}  && {}
}
&&
\xymatrix@R=0.5pc{
{} \ar[dr]^{(\alpha_1,v_1)} \ar[dddr]_(.3){(\alpha_1,v_2)} &&& {} \\
& {\bullet_{v_1}} \ar[r]^{\rho_{13}} \ar[ddr]^(.7){\rho_{14}}
& {\bullet_{v_3}} \ar[ur]^{(v_3, \beta_1)} \ar[dddr]^(.7){(v_3, \beta_l)} \\
{\dots} &&& {\dots} \\
& {\bullet_{v_2}} \ar[r]_{\rho_{24}} \ar[uur]^(.3){\rho_{23}}
& {\bullet_{v_4}} \ar[dr]_{(v_4,\beta_l)} \ar[uuur]_(.7){(v_4,\beta_1)} \\
{} \ar[ur]_{(\alpha_k,v_2)} \ar[uuur]^(.3){(\alpha_k,v_1)} &&& {}
}
\\
(c) && (d)
\end{array}
\]
\caption{Insertion of a poset.
(a) Neighborhood of $v$ in $Q$.
Portion of the quiver $\wt{Q}$ after inserting at $v$:
(b) the chain ${A_n}$; (c) an anti-chain with two elements;
(d) a poset of type $\wt{A}_{2,2}$.}
\label{fig:insert}
\end{figure}

\begin{example}[Insertion of an anti-chain]
Let $S$ consist of two incomparable elements.
The quiver of the insertion $\gL \xleftarrow{v} S$ is obtained from $Q$
by duplicating $v$ and all arrows starting or ending at $v$, as shown
in Figure~\ref{fig:insert}. Each relation in $\gL$ that ends at $v$ is
duplicated to form two relations, one ending at $v_1$ and the other at $v_2$,
and similarly for relations that start at $v$.
There are the commutativity relations
\[
(\alpha_i, v_1)(v_1, \beta_j) - (\alpha_i, v_2)(v_2, \beta_j)
\]
and finally each occurrence of $\alpha_i \beta_j$ inside a path which
is part of a relation is replaced by $(\alpha_i, v_1)(v_1, \beta_j)$.
\end{example}

\begin{example}[Insertion of a poset of type $\wt{A}_{2,2}$] 
\label{ex:insertA22}
Let $S$ be the ordinal sum of two anti-chains of size $2$. Its
incidence algebra equals the path algebra of a quiver of Euclidean type
$\wt{A}_{2,2}$, see also Example~\ref{ex:A22}.
The quiver of the insertion $\gL \xleftarrow{v} S$ is shown
in Figure~\ref{fig:insert}. Each relation in $\gL$ that ends at $v$ is
duplicated to form two relations, one ending at $v_1$ and the other at $v_2$,
and each relation that starts at $v$ is duplicated to form two relations,
one starting at $v_3$ and the other at $v_4$.
There are the commutativity relations
\begin{align*}
(\alpha_i, v_1) \rho_{13} - (\alpha_i, v_2) \rho_{23} &,&
\rho_{13} (v_3, \beta_j) - \rho_{14} (v_4 , \beta_j) ,\\
(\alpha_i, v_1) \rho_{14} - (\alpha_i, v_2) \rho_{24} &,&
\rho_{23} (v_3, \beta_j) - \rho_{24} (v_4 , \beta_j)
\end{align*}
and finally each occurrence of $\alpha_i \beta_j$ inside a path which
is part of a relation is replaced (for instance)
by $(\alpha_i, v_1)\rho_{13}(v_3, \beta_j)$.
\end{example}

\begin{example}[Supercanonical algebras]
Supercanonical algebras were introduced by Lenzing and de la Pe\~{n}a
in~\cite{LenzingdelaPena04} as a generalization of canonical algebras.
Using our Definition~\ref{def:insert}, we can rephrase the original
definition as follows.

Let $t \geq 2$, let $S_1, S_2, \dots, S_t$ be posets and let
$\lambda_3, \dots, \lambda_t$ be $t-2$ pairwise distinct nonzero scalars
from $K$.
The corresponding supercanonical algebra is obtained by
starting with the canonical algebra given by the quiver
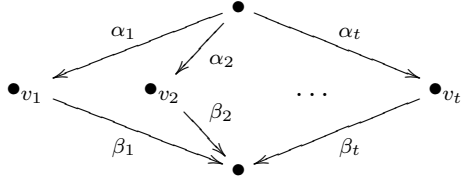
\[
\xymatrix@R=1.5pc@C=1pc{
&&& {\bullet} \ar[dlll]_{\alpha_1} \ar[dl]^{\alpha_2}  \ar[drrr]^{\alpha_t} \\
{\bullet_{v_1}} \ar[drrr]_{\beta_1} &&
{\bullet_{v_2}} \ar[dr]^{\beta_2} && {\dots} &&
{\bullet_{v_t}} \ar[dlll]^{\beta_t} \\
&&& {\bullet}
}
\]
with relations $\alpha_i \beta_i = \alpha_1 \beta_1 - \lambda_i \alpha_2 \beta_2$
for $3 \leq i \leq t$ and iteratively inserting at each vertex $v_i$ 
the poset $S_i$ for $1 \leq i \leq t$.
\end{example}

One can check that Definition~\ref{def:insert} indeed generalizes that of
insertion of posets.

\begin{lemma} \label{l:KXvS}
If $X$ and $S$ are posets and $v \in X$, then
\begin{align*}
K(X \xleftarrow{v} S) \simeq KX \xleftarrow{v} S &,&
(KX)^- = K(X \setminus \{v\}).
\end{align*}
\end{lemma}

The next lemma records the crucial property of the insertion operation
in terms of the Cartan matrices of the algebras involved.
In fact, this is the only property to be used throughout this paper.

\begin{lemma} \label{l:CLvS}
Let $\gL$ be a triangular algebra with quiver $Q$ and Cartan
matrix $C$. Let $S$ be a poset with Cartan matrix $C_S$.
Let $v$ be a vertex in $Q$ and let $\wt{C}$ be the Cartan matrix of $\gL \xleftarrow{v} S$.
Then
\begin{align*}
\wt{C}_{uu'} = C_{uu'} &,&
\wt{C}_{us} = C_{uv} &,&
\wt{C}_{su} = C_{vu} &,&
\wt{C}_{ss'} = (C_S)_{ss'}
\end{align*}
for any $u, u' \in Q_0 \setminus \{v\}$ and $s, s' \in S$.
\end{lemma}
\begin{proof}[Sketch of proof]
Any entry $C_{ij}$ of the Cartan matrix of an algebra $KQ/I$ equals the
dimension of the image modulo $I$ of the space of paths in $Q$ from
$j$ to $i$.

The commutativity-relations~\eqref{it:ins:commS} in
Definition~\ref{def:insert} ensure that $\wt{C}_{ss'} = (C_S)_{ss'}$
for any $s, s' \in S$.
If $u \neq v$, then the relations~\eqref{it:ins:smin} imply that
$\wt{C}_{su} = C_{vu}$ for any $s \in S_{\min}$, and this extends to any
$s \in S$ thanks to the commutativity-relations~\eqref{it:ins:commvout}.
Similar reasoning applies to $\wt{C}_{us}$ using the
relations~\eqref{it:ins:smax} and the
commutativity-relations~\eqref{it:ins:commvin}.
Finally, if $u,u' \neq v$ then $\wt{C}_{uu'} = C_{uu'}$ thanks to the relations~\eqref{it:ins:vthru}
which are well-defined due to the
commutativity-relations~\eqref{it:ins:commvthru}.
\end{proof}

\subsection{The Coxeter polynomial of an insertion}

\begin{theorem} \label{t:refined}
Let $S$ be a poset. There exist polynomials
$\phi^0_S(x), \phi^1_S(x) \in \bZ[x]$ such that
for any triangular algebra $\gL$ and any vertex $v$ in its quiver,
\begin{equation}
\label{e:pLvS}
\phi_{\gL \xleftarrow{v} S}(x) =
\phi_{\gL}(x) \cdot \phi^0_S(x) + \phi_{\gL^-}(x) \cdot \phi^1_S(x) .
\end{equation}
\end{theorem}
\begin{proof}[Idea of proof]
Let $C_S$ be the Cartan matrix of $KS$ (recall that this is the transpose
of the incidence matrix of $S$).
Choose any element $s \in S$ and write $S'=S \setminus \{s\}$.
Let $A$ be the matrix obtained from $xC_S+C_S^T$ by subtracting the
$s$-th row from all the other rows and then subtracting the $s$-th
column from all the other columns. Define
\begin{align} \label{e:refined}
\phi^0_S(x) =
\sum_{\sigma \in \Sigma_{S'}} \sgn(\sigma)
\prod_{i \in S'} A_{i \sigma(i)} &,&
\phi^1_S(x) = 
\sum_{\substack{\sigma \in \Sigma_S \\ \sigma(s) \neq s}} 
\sgn(\sigma) \prod_{i \in S} A_{i \sigma(i)}
\end{align}
where $\Sigma_{S'}, \Sigma_S$ are the groups of permutations on $S'$
and $S$, respectively,
so $\phi^0_S(x)$ is the principal minor of $A$ on the set $S'$.

The Coxeter polynomial of $\gL \xleftarrow{v} S$ equals the determinant
$\det(x\wt{C} + \wt{C}^T)$ where $\wt{C}$ is the Cartan matrix of
$\gL \xleftarrow{v} S$,
hence it can be evaluated using techniques of determinant
expansion; by subtracting its $s$-th row from all other rows
and then the $s$-th column from all other columns,
one obtains~\eqref{e:pLvS} by invoking Lemma~\ref{l:CLvS}.
\end{proof}

\begin{rem}
It is a consequence of the theorem that the polynomials
$\phi^0_S$ and $\phi^1_S$ do not depend on the particular choice of
an element $s \in S$. Moreover, since $\det A$ equals $\det(xC_S + C_S^T)$,
one has $\phi^1_S(x) = \phi_S(x) - (x+1) \phi^0_S(x)$
(compare with~\eqref{e:pSpS0pS1} below), so~\eqref{e:pLvS} can be
restated as
\begin{equation} \label{e:pLvSS}
\phi_{\gL \xleftarrow{v} S}(x) =
\phi_{\gL}(x) \cdot \phi^0_S(x) + \phi_{\gL^-}(x) \cdot \phi_S(x)
- (x+1) \cdot \phi_{\gL^-}(x) \cdot \phi^0_S(x) .
\end{equation}
\end{rem}

\begin{defn}
The polynomials $\phi^0_S$ and $\phi^1_S$ occurring in the theorem are
called the \emph{refined Coxeter polynomials} of $S$.
\end{defn}

\begin{cor} \label{c:pXvS}
Let $S$ be a poset. Then for any poset $X$ and any element $v \in X$,
\begin{equation} \label{e:pXvS}
\phi_{X \xleftarrow{v} S}(x) = \phi_X(x) \cdot \phi^0_S(x) +
\phi_{X \setminus \{v\}}(x) \cdot \phi^1_S(x) .
\end{equation}
\end{cor}
\begin{proof}
Apply Theorem~\ref{t:refined} to the incidence algebra $\gL = KX$ and
use Lemma~\ref{l:KXvS}.
\end{proof}

\subsection{Expressing $\phi^0_S, \phi^1_S$ in terms of
Coxeter polynomials}

Corollary~\ref{c:pXvS} can be used to express the refined Coxeter
polynomials of a poset $S$ in terms of ordinary Coxeter polynomials
of certain posets built from $S$.

Let $\wh{S}$ be the poset obtained from $S$ by adjoining a unique maximal
element, that is, $\wh{S} = S \cup \{\wh{1}\}$ with $s < \wh{1}$ for all
$s \in S$. Then both $S$ and $\wh{S}$ can be written as insertions of $S$
into chains of one and two elements whose Coxeter polynomials are
equal to $x+1$ and $x^2+x+1$, respectively. Namely,
\begin{align*}
S = \{v\} \xleftarrow{v} S &,&
\wh{S} = \left\{v_0 < v_1 \right\} \xleftarrow{v_0} S .
\end{align*}

Applying Corollary~\ref{c:pXvS} then yields the system of equations
\begin{equation}
\label{e:pSpS0pS1}
\begin{pmatrix}
\phi_S(x) \\ \phi_{\wh{S}}(x)
\end{pmatrix}
= 
\begin{pmatrix}
x+1 & 1 \\ x^2+x+1 & x+1
\end{pmatrix}
\begin{pmatrix}
\phi^0_S(x) \\ \phi^1_S(x)
\end{pmatrix}.
\end{equation}
Solving it by
inverting the left factor of the right hand side, we get:
\begin{prop} \label{p:refbyCox}
Let $S$ be a poset. Then
\begin{equation} \label{e:refbyCox}
\begin{pmatrix}
\phi^0_S(x) \\ \phi^1_S(x)
\end{pmatrix}
= 
\frac{1}{x}
\begin{pmatrix}
x+1 & -1 \\ -(x^2+x+1) & x+1
\end{pmatrix}
\begin{pmatrix}
\phi_S(x) \\ \phi_{\wh{S}}(x)
\end{pmatrix}
\end{equation}
where $\wh{S}$ is the poset obtained from $S$ by adjoining a unique
maximal element.
\end{prop}

As a consequence, we obtain:
\begin{cor}
$\phi_{\gL \xleftarrow{v} S}$ depends only on
$\phi_{\gL}$, $\phi_{\gL^-}$, $\phi_S$ and $\phi_{\wh{S}}$.
\end{cor}

\begin{rem}
One could consider instead of $\wh{S}$ the poset obtained from $S$
by adjoining a unique \emph{minimal} element. Since both posets
have the same Coxeter polynomial,
the resulting statements would be identical to those given here.
\end{rem}

Proposition~\ref{p:refbyCox} can be used to calculate the refined
Coxeter polynomials. Table~\ref{tab:refCox} shows them for
some small posets. While $\phi_S$ is always determined by
$\phi^0_S$ and $\phi^1_S$ (see Eq.~\eqref{e:pSpS0pS1}), the last two
rows in the table show that in general $\phi_S$ does not determine
$\phi^0_S$ and $\phi^1_S$; compare with Example~\ref{ex:pXvS}.

\begin{table}
\[
\begin{array}{cccc}
S & \phi_S & \phi^0_S & \phi^1_S \vspace{3pt}\\ \hline
\varnothing & 1 & 0 & 1 \vspace{3pt}\\ 
{\bullet} & x+1 & 1 & 0 \vspace{3pt}\\ 
\xymatrix@=1.5pc{{\bullet} \ar[r] & {\bullet}}
& x^2+x+1 & x+1 & -x \vspace{3pt}\\ 
\xymatrix@=1.5pc{{\bullet} & {\bullet}}
& x^2+2x+1 & 2x+2 & -x^2-2x-1 \vspace{3pt} \\ 
\xymatrix@=1.5pc{{\bullet} \ar[r] & {\bullet} \ar[r] & {\bullet}}
& x^3+x^2+x+1 & x^2+x+1 & -x^2-x \vspace{3pt}\\
\xymatrix@=1.5pc{{\bullet} & {\bullet} \ar[l] \ar[r] & {\bullet}}
& x^3+x^2+x+1 & x^2+2x+1 & -2x^2-2x  
\end{array}
\]
\caption{The Coxeter polynomial $\phi_S$ and the refined Coxeter
polynomials $\phi^0_S$, $\phi^1_S$ for some posets $S$ shown by
their Hasse quivers.}
\label{tab:refCox}
\end{table}


Some self-reciprocal properties of the refined Coxeter polynomials are
listed below.

\begin{prop}
Let $S$ be a poset with $n$ elements. Then:
\begin{enumerate}
\renewcommand{\theenumi}{\alph{enumi}}
\item
$\deg \phi^1_S \leq n$ and $\phi^1_S(x) = x^n \phi^1_S(x^{-1})$;

\item
$\deg \phi^0_S \leq n-1$ and
$\phi^0_S(x) = x^{n-1} \phi^0_S(x^{-1})$.
\end{enumerate}
\end{prop}
\begin{proof}
One can either prove this directly from~\eqref{e:refined},
or use~\eqref{e:refbyCox} and the self-reciprocity
of the Coxeter polynomials $\phi_S$ and $\phi_{\wh{S}}$.
\end{proof}

\subsection{The refined Coxeter polynomials of an insertion}

Proposition~\ref{p:refbyCox} can also be used to derive
formulae for the refined Coxeter polynomials of an insertion of posets,
in analogy with the formula~\eqref{e:pXvS}.

\begin{prop} \label{p:p01XvS}
Let $S$ be a poset. Then for any poset $X$ and any element $v \in X$,
\begin{align}
\label{e:p0XvS}
\phi^0_{X \xleftarrow{v} S}(x) = \phi^0_X(x) \cdot \phi^0_S(x) +
\phi^0_{X \setminus \{v\}}(x) \cdot \phi^1_S(x)
\\
\phi^1_{X \xleftarrow{v} S}(x) = \phi^1_X(x) \cdot \phi^0_S(x) +
\phi^1_{X \setminus \{v\}}(x) \cdot \phi^1_S(x)
\end{align}
\end{prop}
\begin{proof}[Sketch of proof]
The key observation is that the coefficients occurring in
Eq.~\eqref{e:refbyCox} are independent on the particular poset.
Thus one can use Corollary~\ref{c:pXvS} to express $\phi_{X \xleftarrow{v} S}$
and $\phi_{\wh{X} \xleftarrow{v} S}$, then 
verify that $\wh{X} \xleftarrow{v} S \simeq \wh{X \xleftarrow{v} S}$
and $\wh{X} \setminus \{v\} \simeq \wh{X \setminus \{v\}}$
using Lemma~\ref{l:insert}
and finally use Proposition~\ref{p:refbyCox} for each of the three posets
$X$, $X \setminus \{v\}$ and $X \xleftarrow{v} S$.
\end{proof}

\subsection{Cases where $\phi^0_S$ is a Coxeter polynomial}

If the refined Coxeter polynomial $\phi^0_S$ equals the Coxeter polynomial
of some finite-dimensional algebra $\Gamma$, one can use~\eqref{e:pLvSS} in
order to write the Coxeter polynomial of any insertion
$\gL \xrightarrow{v} S$ in terms of ordinary Coxeter polynomials, namely
\[
\phi_{\gL \xleftarrow{v} S}(x) =
\phi_{\gL}(x) \cdot \phi_{\Gamma}(x) + \phi_{\gL^-}(x) \cdot \phi_S(x)
- (x+1) \cdot \phi_{\gL^-}(x) \cdot \phi_{\Gamma}(x) .
\]

An inspection of Table~\ref{tab:refCox} reveals that while
this does not hold for all posets, there are still some posets having
this property. Here is a result in this direction.

\begin{prop} \label{p:p0S1}
Let $S$ be a poset. Then $\phi^0_{\wh{S}}(x) = \phi_S(x)$.
\end{prop}
\begin{proof}
Compare~\eqref{e:p0XvS} for
$X = \xymatrix@1@=1.5pc{{\bullet} \ar[r] & {\bullet}}$
and the first row of~\eqref{e:pSpS0pS1}.
\end{proof}

As a consequence, one can write the Coxeter polynomial of an
insertion of $\wh{S}$ in terms of ordinary Coxeter polynomials.

\begin{cor} \label{cor:insertS1}
For a poset $S$, a triangular algebra $\gL$ and a vertex in $v$ its quiver,
\[
\phi_{\gL \xleftarrow{v} \wh{S}}(x) = 
\phi_{\gL}(x) \cdot \phi_S(x) + \phi_{\gL^-}(x) \cdot \phi_{\wh{S}}(x)
- (x+1) \cdot \phi_{\gL^-}(x) \cdot \phi_S(x) .
\]
\end{cor}

By specializing to posets and slightly changing the notation, we get:
\begin{cor} \label{cor:insertSs}
Assume that $S$ is a poset with a unique maximal element $\wh{1}$.
Then for any poset $X$ and any element $v \in X$,
\[
\phi_{X \xleftarrow{v} S}(x) = 
\phi_{X}(x) \cdot \phi_{S \setminus \{\wh{1}\}}(x) +
\phi_{X \setminus \{v\}}(x) \cdot \phi_S(x) 
-(x+1) \cdot \phi_{X \setminus \{v\}}(x) \cdot \phi_{S \setminus \{\wh{1}\}}(x)
\]
\end{cor}
\begin{proof}
Use Corollary~\ref{cor:insertS1} for the poset $S \setminus \{\wh{1}\}$ and
the algebra $\gL=KX$.
\end{proof}
The above results generalize an analogous formula for the Coxeter polynomial
of a coalescence of trees, which we recall in the next section.

\subsection{Coalescence of trees}

If $T$ is a \emph{forest}, i.e.\ a (finite) graph that is a disjoint union
of trees, then any orientation $\overrightarrow{T}$ of the edges of $T$
gives rise to a quiver which forms the Hasse diagram of a poset
(still denoted $\overrightarrow{T}$) whose
incidence algebra is isomorphic to the path algebra of that quiver and
moreover its derived equivalence class is independent on the
particular orientation. Hence its Coxeter polynomial depends only on
the underlying graph $T$ and we may denote it by $\phi_T(x)$.

A \emph{rooted tree} is a pair $(T,v)$ consisting of a tree $T$ and a
vertex $v \in T$. Denote by $T \setminus \{v\}$ the graph obtained from
$T$ by deleting the vertex $v$ and all the incident edges.
It is a disjoint union of trees.
If $(T_1, v_1)$ and $(T_2, v_2)$ are rooted trees, their
\emph{coalescence} is the rooted tree $(T_1 \cdot T_2, v)$ obtained
from the disjoint union of $T_1$ and $T_2$ by identifying the vertices
$v_1$ and $v_2$ (which become the vertex $v$).
This notion plays a key role in the constructions of isospectral
trees~\cite{Schwenk73}.

The next result expresses the Coxeter polynomial of the coalescence
$T_1 \cdot T_2$ in terms of the Coxeter polynomials of $T_i$ and
$T_i \setminus \{v_i\}$ for $i=1,2$, see~\cite[\S5.2]{Lenzing99}.
\begin{prop} \label{p:coal}
Let $(T_1, v_1)$ and $(T_2, v_2)$ be rooted trees. Then
\[
\phi_{T_1 \cdot T_2}(x) =
\phi_{T_1}(x) \cdot \phi_{T_2 \setminus \{v_2\}}(x) +
\phi_{T_1 \setminus \{v_1\}}(x) \cdot \phi_{T_2}(x)
-(x+1) \cdot \phi_{T_1 \setminus \{v_1\}}(x) \cdot
\phi_{T_2 \setminus \{v_2\}}(x) .
\]
\end{prop}

This result is a special case of~\cite[Theorem~2.2]{Boldt95}.
Alternatively, one can use an analogous formula
involving the characteristic polynomials of the adjacency matrices
occurring in the spectral theory of graphs
\cite[Theorem~2.2.3]{CvRoSi10} and then apply a result of
A'Campo~\cite{ACampo76} (cf.\ Proposition~\ref{p:bipartite} below)
relating these polynomials with the Coxeter polynomials.

We give a new proof based on Corollary~\ref{cor:insertSs}.
First, we note the following:
\begin{lemma}
Let $(T,v)$ be a rooted tree.
\begin{enumerate}
\renewcommand{\theenumi}{\alph{enumi}}
\item
If $\overrightarrow{T}$ is an orientation such that $v$ is a sink or
a source, then the induced orientation on the graph $T \setminus \{v\}$
is the Hasse diagram of the poset $\overrightarrow{T} \setminus \{v\}$.

\item
There exists a unique orientation $\overrightarrow{T}$ such that
$v$ becomes a unique maximal (respectively, minimal) element of the
corresponding poset.
\end{enumerate}
\end{lemma}

\begin{lemma} \label{l:coalHasse}
Let $(T_1, v_1)$ and $(T_2, v_2)$ be rooted trees.
If $\overrightarrow{T_1}$ is the orientation of $T_1$ making $v_1$
into a unique minimal element
and $\overrightarrow{T_2}$ is the orientation of $T_2$ making $v_2$
into a unique maximal element,
then the induced orientation on the coalescence $T_1 \cdot T_2$
is the Hasse diagram of the insertion
$\overrightarrow{T_1} \xleftarrow{v_1} \overrightarrow{T_2}$.
\end{lemma}

Now Proposition~\ref{p:coal} becomes a special case of 
Corollary~\ref{cor:insertSs}. Indeed, consider the orientations
$\overrightarrow{T_1}$ and $\overrightarrow{T_2}$ as in
Lemma~\ref{l:coalHasse}, take the poset $S$ to be $\overrightarrow{T_2}$
with its unique maximal element $v_2$ and the poset $X$
to be $\overrightarrow{T_1}$ with the element $v_1$.

\section{Applications}
\label{sec:Apps}

\subsection{Symmetry for ordinal sums of posets}
Our first application concerns ordinal sums;
we will derive a 
formula for the Coxeter polynomial of an ordinal sum of posets
in terms of their refined Coxeter polynomials and
deduce several consequences.

It is known~\cite[Corollary~4.15]{Ladkani08a}
that if $X_1$ and $X_2$ are two posets,
then the incidence algebras of the ordinal sums $X_1 \oplus X_2$
and $X_2 \oplus X_1$ are derived equivalent, and hence
$\phi_{X_1 \oplus X_2} = \phi_{X_2 \oplus X_1}$. However, there
are examples (e.g.\ \cite[Example~4.20]{Ladkani08a})
showing that for three posets,
the incidence algebras of the ordinal sums
$X_1 \oplus X_2 \oplus X_3$ and $X_2 \oplus X_1 \oplus X_3$ are
in general not derived equivalent. Nevertheless, it is interesting
to ask whether their Coxeter polynomials still coincide.

We start by stating a generalization of Corollary~\ref{c:pXvS} 
and Proposition~\ref{p:p01XvS} for multiple insertions.

\begin{prop} \label{p:multinsert}
Let $Y$ be a poset and $y_1, y_2, \dots, y_n \in Y$ be pairwise distinct
elements. Let $X_1, X_2, \dots, X_n$ be posets.
Then
\[
\phi^{\star}_{(\dots((Y \xleftarrow{y_1} X_1) \xleftarrow{y_2} X_2) \dots)
\xleftarrow{y_n} X_n}(x) =
\sum_{I \subseteq \{1,\dots,n\}}
\phi^{\star}_{Y \setminus \{y_i\}_{i \in I}}(x)
\cdot
\prod_{i=1}^{n} \phi^{\eps_i(I)}_{X_i}(x)
\]
where $\star$ can be $0, 1$ or empty referring to $\phi^0, \phi^1$ or
$\phi$ respectively, and
\[
\eps_i(I) = \begin{cases}
1 & \text{if $i \in I$,} \\
0 & \text{if $i \not \in I$.}
\end{cases}
\]
\end{prop}
\begin{proof}[Sketch of proof]
By induction on $n$, the case $n=1$ being Corollary~\ref{c:pXvS}
or Proposition~\ref{p:p01XvS}.
For the induction step,
one has to verify that insertions ``commute'' with removal of elements,
which is a consequence of Lemma~\ref{l:insert}.
\end{proof}

By viewing an ordinal sum as a multiple insertion according to
Lemma~\ref{l:ordsum}, we derive the following formula for its
Coxeter polynomial and refined Coxeter polynomials.
Denote by $v_m(x)$ the polynomial $(x^m-1)/(x-1)$ for $m \geq 0$.
Set also $v_{-1}(x) = -x^{-1}$.
If $X_1, X_2, \dots, X_n$ is a sequence of posets,
then for any $0 \leq k \leq n$
we can form the following sum running over certain products of their
refined Coxeter polynomials
\[
\Phi_k(x) = \sum_{\substack{\eps \in \{0,1\}^n \\ |\eps|=k}} 
\phi^{\eps_1}_{X_1}(x) \cdot \phi^{\eps_2}_{X_2}(x) \cdot \ldots \cdot
\phi^{\eps_n}_{X_n}(x)
\]
where $|\eps| = \eps_1 + \eps_2 + \dots + \eps_n$ for
$\eps=(\eps_1, \eps_2, \dots, \eps_n) \in \{0,1\}^n$.
Obviously, $\Phi_k(x)$ does not depend on the order of the posets
in the sequence.

\begin{prop} \label{p:ordsum}
Let $X_1, X_2, \dots, X_n$ be posets and let
$X = X_1 \oplus X_2 \oplus \dots \oplus X_n$ be their ordinal sum.
Then
\[
\phi_X(x) = 
\sum_{k=0}^n v_{n+1-k}(x) \cdot \Phi_k(x)
\]
and
\begin{align*}
\phi^0_X(x) = 
\sum_{k=0}^n v_{n-k}(x) \cdot \Phi_k(x), &&
\phi^1_X(x) = 
\sum_{k=0}^n (-x) \cdot v_{n-1-k}(x) \cdot \Phi_k(x) .
\end{align*}
\end{prop}
\begin{proof}
Proposition~\ref{p:multinsert} and Lemma~\ref{l:ordsum} yield the formula
\[
\phi^{\star}_{X_1 \oplus X_2 \oplus \dots \oplus X_n}(x) = 
\sum_{I \subseteq \{1,\dots,n\}}
\phi^{\star}_{\overrightarrow{A_n} \setminus I}(x)
\cdot
\prod_{i=1}^{n} \phi^{\eps_i(I)}_{X_i}(x)
\]
where $\star$ can be $0, 1$ or empty.
The crucial point is that the poset $\overrightarrow{A_n} \setminus I$
depends only on the cardinality of $I$ (namely
$\overrightarrow{A_n} \setminus I \simeq \overrightarrow{A}_{n-|I|}$),
so one can rearrange the summation according to this cardinality.
The explicit expressions for the (refined) Coxeter polynomials
of a chain with $m$ elements
\begin{align*}
\phi_{\overrightarrow{A_m}}(x) = v_{m+1}(x) &,&
\phi^0_{\overrightarrow{A_m}}(x) = v_m(x) &,&
\phi^1_{\overrightarrow{A_m}}(x) = -x \cdot v_{m-1}(x)
\end{align*}
can be verified by induction on $m \geq 0$, using Eq.~\eqref{e:pSpS0pS1} and Proposition~\ref{p:p0S1}. These polynomials for $m \leq 3$ are also shown in
Table~\ref{tab:refCox}.
\end{proof}

We list a few consequences of Proposition~\ref{p:ordsum}.

\begin{cor} \label{cor:ordsum} {\ }
\begin{enumerate}
\renewcommand{\theenumi}{\alph{enumi}}
\item
The Coxeter polynomial and the refined Coxeter polynomials
of an ordinal sum do not depend on the order of summands.

\item
Let $X_1, X_2, \dots, X_n$ and $Y_1, Y_2, \dots, Y_n$ be posets.
Denote their ordinal sums by
$X = X_1 \oplus X_2 \oplus \dots \oplus X_n$ and
$Y = Y_1 \oplus Y_2 \oplus \dots \oplus Y_n$.
If $\phi_{X_i} = \phi_{Y_i}$ and $\phi_{\wh{X_i}} = \phi_{\wh{Y_i}}$
for all $1 \leq i \leq n$, then $\phi_X = \phi_Y$
and $\phi_{\wh{X}} = \phi_{\wh{Y}}$.
\end{enumerate}
\end{cor}

We remark that
as demonstrated by Example~\ref{ex:pXvS}, the conditions
$\phi_{X_i} = \phi_{Y_i}$ alone are not sufficient to imply the
equality $\phi_X = \phi_Y$.

\subsection{New algebras of cyclotomic type}

Another application of refined Coxeter polynomials concerns the construction
of new families of algebras of cyclotomic type.

\begin{defn}
A triangular algebra $\gL$ is of \emph{cyclotomic type} if $\phi_\gL$ is a
product of cyclotomic polynomials.
\end{defn}

A systematic study of algebras of cyclotomic type has been initiated
by de la Pe\~{n}a in the paper~\cite{delaPena14}. This paper surveys
several families of algebras known to be of cyclotomic type. These
include:
\begin{itemize}
\item
The path algebras of Dynkin and Euclidean quivers;

\item
The canonical algebras~\cite{Lenzing99,Ringel84};

\item
Some supercanonical~\cite{LenzingdelaPena04} and
extended canonical~\cite{LenzingdelaPena11} algebras;

\item
Algebras that are fractionally Calabi-Yau; and hence 
by~\cite{HerschendIyama11},

\item
$n$-representation-finite algebras.
\end{itemize}
Moreover, tensor products of algebras of cyclotomic type are of cyclotomic type.

The next lemma shows a multiplicative connection between the Coxeter
polynomial of $\gL \xleftarrow{v} S$ and that of $\gL^-$ when we insert a
poset $S$ with $\phi^0_S = 0$.

\begin{lemma} \label{l:p0S}
Let $S$ be a poset. Then:
\begin{enumerate}
\renewcommand{\theenumi}{\alph{enumi}}
\item
$\phi^0_S(x) = 0
\,\Longleftrightarrow\,
\phi^1_S(x)=\phi_S(x)
\,\Longleftrightarrow\,
\phi_{\wh{S}}(x) = (x+1)\phi_S(x) = \phi_{S+\{\bullet\}}(x)$.

\item
If any of the above equivalent conditions holds, then
$\phi_{\gL \xleftarrow{v} S}(x) = \phi_{\gL^-}(x) \cdot \phi_S(x)$
for any triangular algebra $\gL$ and any vertex $v$ in its quiver.
\end{enumerate}
\end{lemma}
\begin{proof}
The equivalences in the first claim follow from Eq.~\eqref{e:pSpS0pS1},
whereas the second claim follows from Eq.~\eqref{e:pLvS}.
\end{proof}

In view of the previous lemma, it is interesting to find nonempty posets
$S$ with the property that $\phi^0_S=0$. 
It turns out that a collection of such posets is
given by the posets of type $\wt{A}$, which are defined below.

\begin{defn}
A poset $S$ is \emph{of type $\wt{A}$} if its incidence algebra $KS$ is a path
algebra of an orientation of an Euclidean diagram of type $\wt{A}$.
\end{defn}

Let us make the definition more explicit. After fixing a vertex, any
orientation of a cycle on $n \geq 2$ vertices can be encoded by a sequence
of length $n$ consisting of the symbols $+$ and $-$, 
where each symbol tells us whether the corresponding arrow is
directed clockwise ($+$) or anti-clockwise ($-$).
For an orientation that is not a directed cycle,
the encoding sequence can be described in terms of the
lengths of runs of identical consecutive symbols $n^+_1, \dots, n^+_k$
(for runs of $+$-s) and $n^-_1, \dots, n^-_k$ (for runs of $-$-s), where
$k \geq 1$ and the lengths $n^+_i$, $n^-_j$ are positive and sum to $n$.

A poset $S$ is of type $\wt{A}$ if and only if its Hasse quiver is
an orientation of a cycle and $k \geq 2$.
As the Coxeter polynomials of orientations of Euclidean diagrams are
known~\cite[\S5.1]{Lenzing99}, we see that if $S$ is of type $\wt{A}$ then
$\phi_S(x) = (x^p-1)(x^q-1)$ for $p=\sum n^+_i$, $q=\sum n^-_j$,
so in particular its incidence algebra $KS$ is of cyclotomic type.

\begin{example} \label{ex:A22}
The Hasse quiver of the smallest poset of type $\wt{A}$ is given by
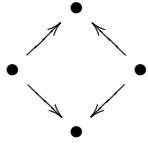
\[
\xymatrix@=1pc{
& {\bullet} \\
{\bullet} \ar[ur] \ar[dr] && {\bullet} \ar[dl] \ar[ul] \\
& {\bullet}
}
\]
which corresponds to the sequence $(+-+-)$ (starting with the leftmost vertex
and going clockwise).
Here $k=2$, $n^+_1=n^+_2=n^-_1=n^-_2=1$ and $p=q=2$, so $\phi_S(x)=(x^2-1)^2$.
One can verify (e.g.\ using Lemma~\ref{l:p0S}) that $\phi^0_S(x)=0$.
\end{example}

The phenomenon observed in the preceding example
occurs for any poset of type $\wt{A}$.

\begin{theorem} \label{t:Atilde}
If $S$ is of type $\wt{A}$ then $\phi^0_S=0$.
\end{theorem}
\begin{proof}[Idea of proof]
We first show that for a poset whose incidence algebra equals the path algebra
of its Hasse quiver, the BGP reflection~\cite{BGP73}
at a vertex (sink or source)
\emph{whose valency equals $2$}
preserves the refined Coxeter polynomials.
By performing an appropriate sequence of such reflections, one is able
to reduce to the case $k=2$ and $n^+_2=n^-_2=1$.

Then we use Lemma~\ref{l:p0XvS}\eqref{it:p0XvX} below
to shorten the lengths $n^+_1$ and $n^-_1$ one has to consider,
so there remains only a few initial cases to check, namely
those with $(n_1^+, n_1^-)$ being equal to one of the pairs
$(1,1), (1,2)$ or $(2,2)$.
\end{proof}

The next lemma shows how to construct new posets with $\phi^0=0$ from older ones.

\begin{lemma} \label{l:p0XvS}
Let $X$ be a poset and let $v \in X$.
\begin{enumerate}
\renewcommand{\theenumi}{\alph{enumi}}
\item
If $\phi^0_{X \setminus \{v\}}=0$ then $\phi^0_{X \xleftarrow{v} S} = 0$
for any poset $S$ satisfying $\phi^0_S=0$.

\item \label{it:p0XvX}
If $\phi^0_{X \setminus \{v\}}=0$ and $\phi^0_X=0$ then
$\phi^0_{X \xleftarrow{v} S} = 0$ for any poset $S$.
\end{enumerate}
\end{lemma}
\begin{proof}
Use Eq.~\eqref{e:p0XvS}.
\end{proof}

These constructions motivate the following definition.

\begin{defn}
Define recursively a class $\cC$ of posets as follows:
\begin{itemize}
\item
The empty poset belongs to $\cC$;

\item
If $X$ is a poset with an element $v \in X$ such that
$X \setminus \{v\}$ belongs to $\cC$, then $X \xleftarrow{v} S$ belongs to $\cC$
for any poset $S$ of type $\wt{A}$.
\end{itemize}
\end{defn}

\begin{rem}
We list a few properties of the class $\cC$.
\begin{itemize}
\item
Any poset of type $\wt{A}$ belongs to $\cC$.

\item
If $X$ and $Y$ are in $\cC$, then $\widetilde{X} \xleftarrow{v} Y$ is also
in $\cC$ for any poset $\widetilde{X}$ and element $v \in \widetilde{X}$
such that $X = \widetilde{X} \setminus \{v\}$.

\item
In particular, $\cC$ is closed under taking ordinal sums
(take $\widetilde{X} = \wh{X}$ and $v=\wh{1}$).
\end{itemize}
\end{rem}

The class $\cC$ provides us with many algebras of cyclotomic type, as
stated in the next corollary.

\begin{cor} \label{cor:cyclo}
Let $S$ be a poset in the class $\cC$. Then:
\begin{enumerate}
\renewcommand{\theenumi}{\alph{enumi}}
\item
$\phi^0_S=0$ and the incidence algebra $KS$ is of cyclotomic type.

\item
If $\gL$ is a triangular algebra and $v$ is a vertex in its quiver,
then $\phi_{\gL \xleftarrow{v} S} = \phi_{\gL^-} \cdot \phi_S$.
In particular, if $\gL^-$ is of cyclotomic type, then so is
$\gL \xleftarrow{v} S$.
\end{enumerate}
\end{cor}
\begin{proof}
The first claim is a consequence of Theorem~\ref{t:Atilde}, Lemma~\ref{l:p0XvS}
and Lemma~\ref{l:p0S}. The second claim is a consequence of the first one
together with Lemma~\ref{l:p0S}.
\end{proof}

\begin{example}
Figure~\ref{fig:posets8} shows 
three posets with 8 elements belonging to the class $\cC$. If $S$ is any
of these posets, then
$\phi^0_S(x)=0$ and $\phi_S(x)=\phi^1_S(x)=(x-1)^4 (x+1)^4$.
We note that the incidence algebras are not derived equivalent,
as their Coxeter transformations are not similar over $\bQ$.
Namely, if $J_n(\lambda)$ denotes a Jordan block of size $n$ with
eigenvalue $\lambda$, then the corresponding Jordan normal forms
are $J_4(1) \oplus J_1(-1)^4$, $J_2(1)^2 \oplus J_1(-1)^4$ and
$J_2(1) \oplus J_1(1)^2 \oplus J_1(-1)^4$.
\end{example}

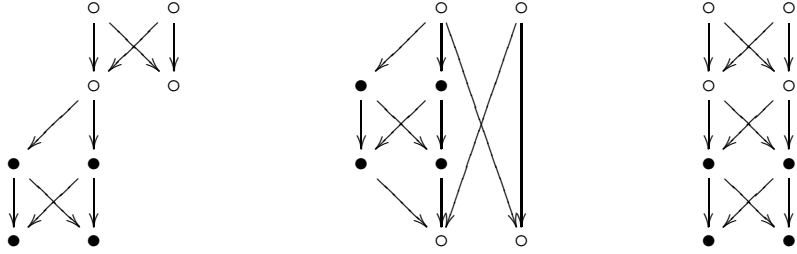
\begin{figure}
\begin{align*}
\xymatrix@=1.5pc{
& {\circ} \ar[d] \ar[dr] & {\circ} \ar[d] \ar[dl] \\
& {\circ} \ar[dl] \ar[d] & {\circ} \\
{\bullet} \ar[d] \ar[dr] & {\bullet} \ar[d] \ar[dl] \\
{\bullet} & {\bullet}
}
&&
\xymatrix@=1.5pc{
& {\circ} \ar[dl] \ar[d] \ar[dddr] & {\circ} \ar[dddl] \ar[ddd] \\
{\bullet} \ar[d] \ar[dr] & {\bullet} \ar[d] \ar[dl] \\
{\bullet} \ar[dr] & {\bullet} \ar[d] \\
& {\circ} & {\circ}
}
&&
\xymatrix@=1.5pc{
{\circ} \ar[d] \ar[dr] & {\circ} \ar[d] \ar[dl] \\ 
{\circ} \ar[d] \ar[dr] & {\circ} \ar[d] \ar[dl] \\ 
{\bullet} \ar[d] \ar[dr] & {\bullet} \ar[d] \ar[dl] \\ 
{\bullet} & {\bullet}
}
\end{align*}
\caption{Hasse diagrams of three posets belonging to the
class $\cC$. All posets are of the form $X \xleftarrow{v} S$
with $S$ and $X \setminus \{v\}$ of type $\wt{A}$.
Vertices of $S$ are colored black and those of $X \setminus \{v\}$
are white.}
\label{fig:posets8}
\end{figure}

\subsection{Interlaced towers of algebras}

The last application we present concerns interlaced towers of algebras
introduced by de la Pe\~{n}a in the paper~\cite{delaPena15}. We first
recall the notion of representable polynomials, introduced
by Lenzing and de la Pe\~{n}a in~\cite{LenzingdelaPena09}.

\begin{defn}[\protect{\cite{LenzingdelaPena09}}]
A polynomial $\phi \in \bZ[x]$ is \emph{represented by} a polynomial
$q \in \bZ[y]$ if
\begin{equation} \label{e:represent}
\phi(x^2) = x^{\deg q} \cdot q(x+x^{-1}).
\end{equation}

A polynomial is \emph{representable} if it is represented by some polynomial.
\end{defn}

\begin{example}
The polynomial $x+1$ is represented by $y$, since $x^2+1 = x^1 (x+x^{-1})$.
\end{example}

Eq.~\eqref{e:represent} implies a strong relation between the roots of
the polynomials $\phi$ and $q$, as demonstrated by the next lemma.

\begin{lemma}[\protect{\cite{LenzingdelaPena09}}]
Assume that $\phi$ is represented by $q$. Then:
\begin{enumerate}
\renewcommand{\theenumi}{\alph{enumi}}
\item
$\Roots(\phi) \subset \bS^1$ if and only if $\Roots(q) \subset [-2,2]$.

\item
$\Roots(\phi) \subset \bS^1 \cup \bR_{>0}$ 
if and only if $\Roots(q) \subset \bR$.
\end{enumerate}
\end{lemma}

The motivation for considering representable polynomials originates
in the following result of A'Campo. Recall that a quiver is
\emph{bipartite} if every vertex is a source or a sink.

\begin{prop}[\cite{ACampo76}] \label{p:bipartite}
If $Q$ is a bipartite quiver, then $\phi_Q$ is representable by the
characteristic polynomial of the adjacency matrix $A_{\overline{Q}}$
of the underlying graph of $Q$.
\end{prop}

This allows the use of methods from spectral graph theory~\cite{CvRoSi10},
and in particular the interlacing properties of eigenvalues of symmetric
matrices, for assessing spectral properties of Coxeter polynomials, see for
instance~\cite[Proposition~2.6]{delaPena94}.
Since the matrix $A_{\overline{Q}}$ is symmetric, its eigenvalues
are real and hence the roots of $\phi_Q$ lie in $\bS^1 \cup \bR_{>0}$.

\medskip

We first make the observation that any self-reciprocal polynomial is
representable.
\begin{prop}
Assume that $p \in \bZ[x]$ satisfies $\deg p \leq n$ and $p(x)=x^n p(x^{-1})$.
Then there exists $q \in \bZ[y]$ with $\deg q \leq n$ such that
$p(x^2) = x^n q(x+x^{-1})$.
\end{prop}
\begin{proof}[Sketch of proof]
We proceed by induction on $n$, the cases $n=0,1$ being clear.

Let $a_n$ be the coefficient of $x^n$ in $p(x)$. Then
$p(x^2) - a_n x^n (x + x^{-1})^n$ equals $x^2 p'(x^2)$ for a polynomial $p'$
with $\deg p' \leq n-2$ and $p'(x) = x^{n-2} p'(x^{-1})$. 
If $q'$ is the polynomial in $\bZ[y]$ corresponding to $p'$ under the
induction hypothesis, then $q(y) = a_n y^n + q'(y)$ corresponds to $p$.
\end{proof}

One can write an explicit formula relating the coefficients of $p$ and $q$,
but it will not be needed in what follows.
It suffices to note that from the above proof we see that the
coefficient of $y^n$ in $q(y)$ equals that of $x^n$ in $p(x)$.
In particular, if $\deg p = n$ then also $\deg q=n$.

\begin{cor}
If $p(x) = x^{\deg p} \cdot p(x^{-1})$ then $p$ is representable by a
polynomial $q$ with $\deg q = \deg p$. If $p$ is monic, then so is $q$.
\end{cor}

\begin{cor} \label{cor:represent}
Let $\gL$ be a finite-dimensional algebra.
\begin{enumerate}
\renewcommand{\theenumi}{\alph{enumi}}
\item
If $\det C_{\gL} \neq 0$ and the matrix $C_{\gL}^{-1} C_{\gL}^T$ has
integer entries, then the Coxeter polynomial $\phi_{\gL}$ is representable
by a monic polynomial.

\item
If $\gL$ is Iwanaga-Gorenstein with $\det C_{\gL} \neq 0$, then
$\phi_{\gL}$ is representable by a monic polynomial.

\item
In particular, Coxeter polynomials of triangular algebras or more
generally algebras of finite global dimension are 
representable by monic polynomials.
\end{enumerate}
\end{cor}

In view of the above corollary, we can reformulate the
definition in~\cite{delaPena15} as follows.

\begin{defn}[\protect{\cite[\S5.2]{delaPena15}}] \label{def:tower}
A sequence $\gL_0, \gL_1, \dots, \gL_n$ of triangular algebras
is an \emph{interlaced tower of algebras} if 
$n \geq 2$ and the following conditions hold:
\begin{enumerate}
\renewcommand{\theenumi}{\roman{enumi}}
\item \label{it:degree}
$\deg \phi_{\gL_{i+1}} = \deg \phi_{\gL_i} + 1$
for all $0 \leq i < n$;

\item \label{it:interlace}
$\phi_{\gL_{i+1}}(x) =
(x+1) \cdot \phi_{\gL_i}(x) - x \cdot \phi_{\gL_{i-1}}(x)$
for all $1 \leq i < n$;

\item \label{it:simple}
The roots of the polynomial representing $\phi_{\gL_0}$ are
real and simple.
\end{enumerate}
\end{defn}

If $q_i$ is the polynomial representing $\phi_{\gL_i}$, then the linear
recurrence~\eqref{it:interlace} implies that
$q_{i+1}(y) = y q_i(y) - q_{i-1}(y)$, hence
a version of Sturm theorem as cited in~\cite{delaPena15,LenzingdelaPena09}
seems to imply an interlacing property for the real roots of
$q_i(y)$ and $q_{i+1}(y)$.

Now observe that the coefficients occurring in condition~\eqref{it:interlace}
are the refined Coxeter polynomials of a poset, namely
$\phi^0_{\bullet \to \bullet}(x) = x+1$
and $\phi^1_{\bullet \to \bullet}(x) = -x$, see Table~\ref{tab:refCox}.
This is the key point behind our next observation.
Note that a description of the quiver with relations of the algebra
$\gL \xleftarrow{v} \overrightarrow{A_i}$ in terms of those of $\gL$
is given in Example~\ref{ex:insertAn}.

\begin{prop} \label{p:tower}
Let $\gL$ be a triangular algebra and $v$ be any vertex in its quiver.
Then the sequence of triangular algebras defined by
$\gL_0 = \gL^-$, $\gL_1 = \gL$ and
$\gL_i = \gL \xleftarrow{v} \overrightarrow{A_i}$
for $i \geq 1$ satisfies conditions~\eqref{it:degree} and~\eqref{it:interlace}
in the above definition.
\end{prop}
\begin{proof}[Sketch of proof]
Fix $i \geq 1$ and consider the poset which is a chain of the $i+1$
elements $v_1 < v_2 < \dots < v_{i+1}$. Then we can write
\begin{align*}
\gL_{i+1} = \gL \xleftarrow{v} \{v_1 < v_2 < \dots < v_i < v_{i+1}\}
&& \text{and} &&
\gL_i = \gL \xleftarrow{v} \{v_1 < v_2 < \dots < v_i\} ,
\end{align*}
so if we consider insertion at the vertex $v_i$ of $\gL_i$, then
$\gL_{i+1} \simeq \gL_i \xleftarrow{v_i} \{v_i < v_{i+1}\}$ and
$\gL_i^- = \gL_{i-1}$.
The claim now follows from Theorem~\ref{t:refined} for the algebra $\gL_i$,
the vertex $v_i$ and the poset $\{v_i < v_{i+1}\}$.
\end{proof}

The paper~\cite{delaPena15} presents several claims concerning the Mahler
measures of the Coxeter polynomials of interlaced towers of algebras.

\begin{defn}
The \emph{Mahler measure} of a self-reciprocal polynomial
$\phi \in \bZ[x]$ is defined as
$M(\phi) = \prod_{\lambda \in \Roots(\phi)} \max(1,|\lambda|)$.
\end{defn}

Note that $M(\phi) \geq 1$, and due to Kronecker's Theorem,
equality holds if and only if $\phi$ is a product of cyclotomic polynomials.

In order to capture the essence of the claims made in~\cite{delaPena15},
let us formulate them for the shortest possible interlaced towers, i.e.\
those consisting of three algebras. In what follows, $\gL, \gL', \gL''$
denotes an interlaced tower of algebras with 
Coxeter polynomials $\phi, \phi', \phi''$ and Mahler measures
$M = M(\phi)$, $M' = M(\phi')$, $M'' = M(\phi'')$.

\begin{prop}[\protect{\cite[Theorem~2, Theorem~5.5]{delaPena15}}]
Assume that $\Roots(\phi'') \subset \bS^1 \cup \bR_{>0}$.
Then the following assertions hold:
\begin{enumerate}
\renewcommand{\theenumi}{\alph{enumi}}
\item
$\Roots(\phi) \subset \bS^1 \cup \bR_{>0}$ and
$\Roots(\phi') \subset \bS^1 \cup \bR_{>0}$.

\item \label{it:M1}
If $M''=1$ then $M=M'=1$.

\item \label{it:Mprod}
If $M''>1$ then $M, M' < M''$ and $M'' \leq M \cdot M'$.
\end{enumerate}
\end{prop}

We give counterexamples to the above claims~\eqref{it:M1}
and~\eqref{it:Mprod}. In all the examples we 
start with a triangular algebra $\gL$, choose a vertex $v$
and construct the algebras $\gL', \gL''$ as insertions of
chains
$\gL' = \gL \xleftarrow{v} \{v_1 < v_2\}$ and
$\gL'' = \gL \xleftarrow{v} \{v_1 < v_2 < v_3\}$.
The conditions~\eqref{it:degree} and~\eqref{it:interlace}
in Definition~\ref{def:tower}
are guaranteed to hold thanks to Proposition~\ref{p:tower}.
The roots of the Coxeter polynomials $\phi, \phi', \phi''$
will all lie inside $\bS^1 \cup \bR_{>0}$ since the algebras
constructed are either (derived equivalent to) path algebras of trees or 
their Coxeter polynomials are products of cyclotomic polynomials.
Condition~\eqref{it:simple} is then verified by checking in each case
that the polynomial representing $\phi$ and its derivative
are relatively prime using the
\textsc{Magma} computer algebra system~\cite{MAGMA}, so
$\gL, \gL', \gL''$ forms an interlaced tower of algebras.

An alternative way to verify condition~\eqref{it:simple} is to note that
in each of the examples, the algebra $\gL$ is (derived equivalent to)
a path algebra of a tree $T$ with $9$ vertices, hence by
Proposition~\ref{p:bipartite} the Coxeter polynomial $\phi_\gL$ is
represented by the characteristic polynomial of the adjacency matrix of $T$.
A list of all the trees with at most $9$ vertices together with their
characteristic polynomials and their roots can be found in Table~A4
in the appendix of the textbook~\cite{CvRoSi10}.

In each example we give the algebra $\gL$ as quiver with relations
(if any) together with the vertex $v$, which is colored white.
Non-oriented edges in the quiver can be oriented arbitrarily.

\begin{example}
$M'' = 1$ but $M, M' > 1$.

Let $\gL$ be the extended canonical algebra~\cite{LenzingdelaPena11}
with weight sequence $(2,3,4)$ whose quiver is
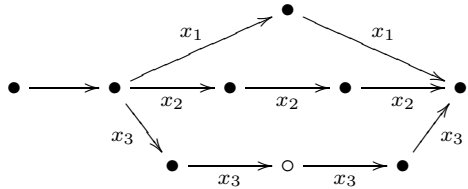
\[
\xymatrix@R=1.5pc@C=0.8pc{
&& &&& {\bullet} \ar[drrr]^{x_1} \\
{\bullet} \ar[rr] &&
{\bullet} \ar[urrr]^{x_1} \ar[rr]_{x_2} \ar[dr]_{x_3} &&
{\bullet} \ar[rr]_{x_2} && {\bullet} \ar[rr]_{x_2} && {\bullet} \\
&&& {\bullet} \ar[rr]_{x_3} && {\circ} \ar[rr]_{x_3} &&
{\bullet} \ar[ur]_{x_3}
}
\]
with the relation $x_1^2-x_2^3+x_3^4=0$.
Then $\gL',\gL''$ are the extended canonical algebras with
weights $(2,3,5), (2,3,6)$,
and $M \approx 1.281, M' \approx 1.176, M''=1$.
Note that $\gL,\gL'$ are derived equivalent to the path algebras
of the stars with arm lengths $(2,4,5), (2,3,7)$
and $\gL''$ is derived equivalent to the canonical algebra
with weights $(2,3,7)$, see
\cite[Propositions~3.3 and~3.5]{LenzingdelaPena11}.
The Coxeter polynomial $\phi_\gL$ is represented by the polynomial
$y^9 - 8y^7 + 20y^5 - 17y^3 + 3y$
(no.~91 in~\cite[Table~A4]{CvRoSi10}).
\end{example}

\begin{example}
$M''>1$ but $M > M' > M''$.

Let $\gL$ be the path algebra of the tree
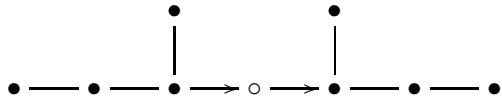
\[
\xymatrix@=1.5pc{
& & {\bullet} \ar@{-}[d] & & {\bullet} \ar@{-}[d] \\
{\bullet} \ar@{-}[r] & {\bullet} \ar@{-}[r] &
{\bullet} \ar[r] & {\circ} \ar[r] & 
{\bullet} \ar@{-}[r] & {\bullet} \ar@{-}[r] & {\bullet}
}
\]
Then $\gL, \gL', \gL''$ are path algebras of trees and
$M \approx 1.722$, $M' \approx 1.640$, $M'' \approx 1.582$.
The Coxeter polynomial $\phi_\gL$ is represented by the polynomial
$y^9 - 8y^7 + 19y^5 - 14y^3 + 3y$
(no.~85 in~\cite[Table~A4]{CvRoSi10}).
\end{example}

\begin{example}
$M < M' < M''$ but $M'' > M M'$.

Let $\gL$ be the path algebra of the tree (Euclidean diagram $\wt{E}_8$)
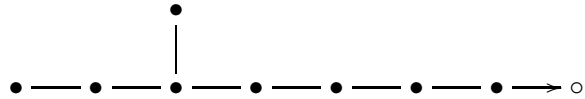
\[
\xymatrix@=1.5pc{
& & {\bullet} \ar@{-}[d] \\
{\bullet} \ar@{-}[r] & {\bullet} \ar@{-}[r] & {\bullet} \ar@{-}[r] &
{\bullet} \ar@{-}[r] & {\bullet} \ar@{-}[r] & {\bullet} \ar@{-}[r] &
{\bullet} \ar[r] & {\circ}
}
\]
Then $\gL, \gL', \gL''$ are the path algebras of stars with arm lengths
$(2,3,6), (2,3,7), (2,3,8)$ and $M = 1$, $M' \approx 1.176$,
$M'' \approx 1.230$.
The Coxeter polynomial $\phi_\gL$ is represented by the polynomial
$y^9 - 8y^7 + 20y^5 - 17y^3 + 4y$
(no.~92 in~\cite[Table~A4]{CvRoSi10}).
\end{example}

\bibliographystyle{amsplain}
\bibliography{refined}

\end{document}